\documentclass[11pt]{article}
\usepackage[utf8]{inputenc}
\usepackage[T1]{fontenc}             
\usepackage{lmodern}
\usepackage[french,english]{babel}
\usepackage{amsmath}
\usepackage{amsthm}
\usepackage{amsmath}
\usepackage{amssymb}
\usepackage{mathrsfs}
\usepackage{geometry}
\usepackage{graphicx}
\usepackage{comment}
\usepackage{url}
\usepackage{titlesec}
\usepackage[all]{xy}
\usepackage{enumitem}
\usepackage{amssymb,amsmath,amsfonts,amsxtra,amsthm}
\usepackage[all]{xy}
\usepackage{stmaryrd} 
\usepackage{tikz}

\usepackage[colorlinks=red]{hyperref}

\newtheorem{df}{Definition}[section]
\newtheorem{rem}[df]{Remark}

\newtheorem{ex}[df]{Example}

\newtheorem{thm}[df]{Theorem}
\newtheorem{pp}[df]{Proposition} 
\newtheorem{lm}[df]{Lemma}

\newtheorem{cor}[df]{Corollary}

\newtheorem*{notation}{Notation}
\newtheorem{thm intro}{Theorem}

\def\so{\mathfrak{so}}
\def\su{\mathfrak{su}}

\def\sl{\mathfrak{sl}}

\def\sp{\mathfrak{sp}}

\def\gg{\mathfrak{g}}
\def\hh{\mathfrak{h}}

\def\mm{\mathfrak{m}}

\def\uu{\mathfrak{u}}

\def\kt{\tilde{k}}
\def\gt{\tilde{\gg}}

\def\gl{\mathfrak{gl}}

\title{\Large{\textbf{The Kostant invariant and special $\epsilon$-orthogonal representations for $\epsilon$-quadratic colour Lie algebras }}}
\author{Philippe Meyer}
\date\today

\begin{document}
\maketitle
\begin{center}
\textbf{Abstract} 
\end{center}
Let $k$ be a field of characteristic not two or three, let $\gg$ be a finite-dimensional colour Lie algebra and let $V$ be a finite-dimensional representation of $\gg$. In this article we give various ways of constructing a colour Lie algebra $\gt$ whose bracket in some sense extends both the bracket of $\gg$ and the action of $\gg$ on $V$. Colour Lie algebras, originally introduced by R. Ree (\cite{Ree60}), generalise both Lie algebras and Lie superalgebras, and in those cases our results imply many known results (\cite{Kos99}, \cite{Kos01}, \cite{ChenKang2015}, \cite{StantonSlupinski15}). For a class of representations arising in this context we show there are covariants satisfying identities analogous to Mathews identities for binary cubics.

\section*{Introduction}
\addcontentsline{toc}{section}{\protect\numberline{}Introduction}

In \cite{Kos56}, \cite{Kos99}, B. Kostant studied the following problem:  if $(\gg,(\phantom{v},\phantom{v})_{\gg})$ is a complex quadratic Lie algebra and $(V,(\phantom{v},\phantom{v})_V)$ is an orthogonal representation, when is there a Lie bracket on $\gt=\gg\oplus V$ extending the bracket of $\gg$ and the action of $\gg$ on $V$ and such that $(\phantom{v},\phantom{v})_{\gg}\perp (\phantom{v},\phantom{v})_V$ is $\gt$-invariant ? He first observed that the component in $\gg$ of such a bracket is essentially the moment map $\mu : \Lambda^2(V) \rightarrow \gg$ of the representation, and that the component in $V$ of this bracket defines an alternating, $\gg$-invariant trilinear form $\phi$ on $V$. He then gave a necessary and sufficient condition for the moment map of an orthogonal representation and an alternating invariant trilinear form to arise in this way. This condition is the vanishing of an invariant taking values in the Clifford algebra $C(V,(\phantom{v},\phantom{v})_V)$. Furthermore, in the case where $\phi$ is non-zero and where $\gg\oplus V$ is a Lie algebra, Kostant (\cite{Kos99}) constructed an element of $U(\gg\oplus V)\otimes C(V,(\phantom{v},\phantom{v})_V)$, the cubic Dirac operator, and proved an analogue of the Parthasarathy formula (\cite{Parthasarathy72}).
\vspace{0.2cm}

Later, Kostant (\cite{Kos01}) also considered Lie superalgebras from this point of view and extended his results to symplectic representations of quadratic Lie algebras arising from the canonical $\mathbb{Z}_2$-gradation of quadratic Lie superalgebras. In that case, the relevant invariant takes its values in a Weyl algebra as opposed to a Clifford algebra. Z. Chen and Y. Kang generalised these results to orthosymplectic representations of quadratic Lie superalgebras where it turns out that the relevant invariant takes its values in a Clifford-Weyl algebra (\cite{ChenKang2015}).
\vspace{0.2cm}

A different situation is considered in \cite{CS09}, \cite{StantonSlupinski15} where the authors study the existence of a Lie bracket on $\gt=\gg\oplus \sl(2,k)\oplus V\otimes k^2$ extending the bracket of a Lie algebra $\gg$ defined over a field $k$ of characteristic not two or three and its action on a symplectic representation $V$. A necessary and sufficient condition for this to be the case can be expressed in terms of the moment map of $V$, and symplectic representations whose moment map satisfy this condition are called special symplectic.
\vspace{0.2cm}

In this article we will consider the extension problems above for colour Lie algebras and their representations over a field $k$ of characteristic not two or three. Recall that for multilinear maps on vector spaces graded by an abelian group $\Gamma$ there is a notion of symmetry and antisymmetry with respect to any commutation factor $\epsilon$ of $\Gamma$. A $(\Gamma,\epsilon)$-colour Lie algebra (abbreviated to colour Lie algebra in what follows) is a $\Gamma$-graded vector space together with a bracket which is antisymmetric in this sense and satisfies an appropriate Jacobi identity. The natural generalisation of an orthogonal or symplectic representation of a quadratic Lie algebra (as well as of an orthosymplectic representation of a quadratic Lie superalgebra) is an $\epsilon$-orthogonal representation of an $\epsilon$-quadratic colour Lie algebra (see Example \ref{ex rep osp o et sp}). Our first theorem (Section \ref{section colour Lie type}) is:

\begin{thm intro} \label{thm intro colour Lie type}
Let $\epsilon$ be a commutation factor of an abelian group $\Gamma$. Let $\rho : \gg \rightarrow \so_{\epsilon}(V,(\phantom{v} , \phantom{v})_V)$ be a finite-dimensional $\epsilon$-orthogonal representation of a finite-dimensional $\epsilon$-quadratic colour Lie algebra $(\gg,B_{\gg})$ and let $\mu \in Alt_{\epsilon}^2(V,\gg)$ be its moment map.
\begin{enumerate}[label=\alph*)]
\item If there exists an $\epsilon$-quadratic colour Lie algebra structure on $(\gg\oplus V,B_{\gg}\perp (\phantom{v} , \phantom{v})_V)$ extending the bracket of $\gg$ and the action of $\gg$ on $V$, then $\phi \in Alt^2_{\epsilon}(V,\gg\oplus V)$ defined by
$$\phi(v,w)=\lbrace v,w \rbrace - \mu(v,w) \qquad \forall v,w\in V$$
is of degree $0$, takes its values in $V$ and satisfies:
\begin{align}
\rho(x)(\phi(v,w))&=\phi(\rho(x)(v),w)+\epsilon(x,v)\phi(v,\rho(x)(w)) \quad &\forall x \in \gg, ~ \forall v,w \in V \label{INTRO phi g-inv}, \\
(\phi(u,v),w)_V&=-\epsilon(u,v)(v,\phi(u,w))_V \quad &\forall u,v,w \in V \label{INTRO phi ()-inv}.
\end{align}
\item \label{INTRO condition b thm 1} Let $\phi \in Alt_{\epsilon}^2(V,V)$ be of degree $0$ and satisfy \eqref{INTRO phi g-inv} and \eqref{INTRO phi ()-inv}. Let $\gt:=\gg\oplus V$, let $B_{\gt}:=B_{\gg}\perp (\phantom{v},\phantom{v})_V$ and let $\lbrace \phantom{v},\phantom{v} \rbrace \in Alt^2_{\epsilon}(\gt,\gt)$ be the unique map which extends the bracket of $\gg$, the action of $\gg$ on $V$ and such that
$$\lbrace v,w\rbrace=\mu(v,w)+\phi(v,w) \qquad \forall v,w \in V.$$
Then the following are equivalent:
\begin{enumerate}[label=\roman*)]
\item $(\gt,B_{\gt},\lbrace \phantom{v},\phantom{v} \rbrace)$ is an $\epsilon$-quadratic colour Lie algebra.
\item $(\mu +\phi)\wedge_{B_{\gt}}(\mu+\phi)=0$.
\end{enumerate}
\end{enumerate}
\end{thm intro}

In the case of Lie algebras and Lie superalgebras, the invariant $(\mu +\phi)\wedge_{B_{\gt}}(\mu+\phi)$, which takes its values in an $\epsilon$-exterior algebra, is the same after quantisation as the invariants of \cite{Kos99}, \cite{Kos01}, \cite{ChenKang2015} which take their values in an $\epsilon$-Clifford algebra. The main tools involved in the proof of Theorem \ref{thm intro colour Lie type} are multilinear algebra over vector spaces graded by an abelian group (Section \ref{section multilinear}) and what we call, by a slight abuse of language, the moment map of an $\epsilon$-orthogonal representation of an $\epsilon$-quadratic colour Lie algebra (Section \ref{section moment map}). Following Kostant's original terminology, a representation satisfying conditions \ref{INTRO condition b thm 1} of Theorem \ref{thm intro colour Lie type} is called of colour Lie type and, if in addition $\phi\equiv 0$, of colour $\mathbb{Z}_2$-Lie type.
\vspace{0.2cm}

In Section \ref{section bianchi}, we give a different interpretation of the data $(\mu,\phi)$ of this theorem in terms of ``curvature tensors'' and show that the condition $(\mu + \phi)\wedge_{B_{\gt}} (\mu + \phi)=0$ is equivalent to a ``second Bianchi identity''. This point of view suggests a natural algebraic condition on $\epsilon$-orthogonal representations which leads to the notion of special $\epsilon$-orthogonal representation. It turns out that there is a link between special $\epsilon$-orthogonal representations and $\epsilon$-orthogonal representations of colour $\mathbb{Z}_2$-Lie type:

\begin{thm intro}
Let $\epsilon$ be a commutation factor of an abelian group $\Gamma$ such that the representation $k^2$ of $\sl(2,k)$ is an $\epsilon$-orthogonal representation with respect to $(\Gamma,\epsilon)$. Let $\gg \rightarrow \so_{\epsilon}(V,(\phantom{v},\phantom{v})_V)$ be a finite-dimensional faithful $\epsilon$-orthogonal representation of a finite-dimensional $\epsilon$-quadratic colour Lie algebra and let $W$ be a finite-dimensional $\Gamma$-graded vector space together with a non-degenerate $\epsilon$-symmetric bilinear form $(\phantom{v},\phantom{v})_W$. Then the $\epsilon$-orthogonal representation
$$\gg\oplus \so_{\epsilon}(W,(\phantom{v},\phantom{v})_W) \rightarrow \so_{\epsilon}(V\otimes W,(\phantom{v},\phantom{v})_{V\otimes W})$$
is of colour $\mathbb{Z}_2$-Lie type if and only if one of the following holds:
\begin{enumerate}[label=\alph*)]
\item $\gg$ is isomorphic to $\so_{\epsilon}(V,(\phantom{v},\phantom{v})_V)$ ;
\item $dim(W)=1$, $(\phantom{v},\phantom{v})_W$ is symmetric and $\gg \rightarrow \so_{\epsilon}(V,(\phantom{v},\phantom{v})_V)$ is of colour $\mathbb{Z}_2$-Lie type ;
\item $dim(W)=2$, $(\phantom{v},\phantom{v})_W$ is antisymmetric and $\gg \rightarrow \so_{\epsilon}(V,(\phantom{v},\phantom{v})_V)$ is special $\epsilon$-orthogonal.
\end{enumerate}
\end{thm intro}

It follows from this theorem that a special $\epsilon$-orthogonal representation $\gg\rightarrow\so_{\epsilon}(V,(\phantom{v},\phantom{v})_V)$ can be extended to a colour Lie algebra of the form
$$\gt=\gg\oplus\sl(2,k)\oplus V\otimes k^2.$$
In this way, special symplectic representations of Lie algebras give rise to Lie algebras, and special orthogonal representations of Lie algebras give rise to Lie superalgebras. The special symplectic representations of Lie algebras considered in \cite{StantonSlupinski15} are special $\epsilon$-orthogonal in this sense, and examples of orthogonal representations of Lie algebras which are special $\epsilon$-orthogonal in this sense are (see chapter $5$ of \cite{MeyerThesis}):
\begin{itemize}
\item[$\bullet$] a one-parameter family $V_{\alpha}$ of $4$-dimensional representations of $\sl(2,k)\oplus\sl(2,k)$ ;
\item[$\bullet$] the $7$-dimensional fundamental representation of a Lie algebra of type $G_2$ ;
\item[$\bullet$] the $8$-dimensional spinor representation of a Lie algebra of type $\so(7)$ where this representation is defined over $k$.
\end{itemize}
The associated Lie superalgebras $\gt$ are respectively exceptional simple Lie superalgebras of type $D(2,1;\alpha), G_3$ and $F_4$. Certain representations considered in \cite{Elduque06} are special $\epsilon$-orthogonal representations if we extend this notion to fields of characteristic three. A link between symplectic triple systems and special symplectic representations of Lie algebras is detailed in \cite{StantonSlupinski15} and a similar link between orthogonal triple systems and special orthogonal representations probably exists (compare \cite{OkuboNoriaki04} and chapter $5$ of \cite{MeyerThesis}).
\vspace{0.1cm}

Finally in Section \ref{section covariants} we study geometric properties of special $\epsilon$-orthogonal representations. It is well-known that the space of binary cubics, a special symplectic representation of $\sl(2,k)$, admits three covariants and that these covariants satisfy remarkable identities (\cite{Eisenstein1844}, \cite{Mathews11}). More generally, special symplectic representations of Lie algebras admit three covariants which are polynomial functions on the representation space and these covariants satisfy generalised Mathews identities (\cite{StantonSlupinski15}). Analogously, we define three covariants of special $\epsilon$-orthogonal representations, one of which is the moment map, and prove corresponding Mathews identities. Note that in order to formulate these identities we have to define a notion of composition for $\epsilon$-alternating multilinear forms extending the composition of symmetric multilinear forms.

\begin{thm intro}
Let $\epsilon$ be a commutation factor of an abelian group $\Gamma$. Let $\rho : \gg \rightarrow \so_{\epsilon}(V,(\phantom{v},\phantom{v}))$ be a finite-dimensional special $\epsilon$-orthogonal representation of a finite-dimensional $\epsilon$-quadratic colour Lie algebra and let $\mu \in Alt^2_{\epsilon}(V,\gg)$, $\psi \in Alt_{\epsilon}^3(V,V)$ and $Q \in Alt_{\epsilon}^4(V)$ be its covariants. We have the following identities:
\begin{alignat*}{4}
a)&\qquad\qquad\qquad\qquad&\mu\wedge_{\rho} \psi&=-\frac{3}{2}Q\wedge_{\times} Id_V \quad &&\text{ in } Alt^5_{\epsilon}(V,V),\\
b)&\qquad\qquad\qquad\qquad&\mu \circ \psi &= 3Q \wedge_{\times} \mu \quad &&\text{ in } Alt^6_{\epsilon}(V,\gg),\\
c)&\qquad\qquad\qquad\qquad&\psi\circ \psi&=-\frac{27}{2} Q\wedge_{\epsilon} Q \wedge_{\times} Id_V \quad &&\text{ in } Alt^9_{\epsilon}(V,V),\\
d)&\qquad\qquad\qquad\qquad& Q\circ \psi&=-54 Q\wedge_{\epsilon} Q \wedge_{\epsilon} Q \quad && \text{ in } Alt^{12}_{\epsilon}(V,k).
\end{alignat*}
\end{thm intro}

\section*{Acknowledgements}

The author wants to express his gratitude to Marcus J. Slupinski for his suggestions, his encouragement and his expert advice.
\vspace{0.5cm}

\textit{Throughout this paper, the field $k$ is of characteristic not two or three.}

\section{Multilinear algebra associated to $\Gamma$-graded vector spaces} \label{section multilinear}

\subsection{Vector spaces and algebras graded by an abelian group}
In this subsection, we give definitions and examples of vector spaces and algebras graded by an abelian group. Let $\Gamma$ be an abelian group.

\begin{df}
A vector space $V$ with a decomposition $V=\bigoplus \limits_{\gamma \in \Gamma} V_{\gamma}$ is said to be $\Gamma$-graded and an element $v\in V_{\gamma}$ is said to be homogeneous.
\end{df}

For an element $v \in V_{\gamma}$ we set $|v|:=\gamma$ and we call $|v|$ the degree of $v$. For convenience, whenever the degree of an element is used in a formula, it is assumed that this element is homogeneous and that we extend by linearity the formula for non-homogeneous elements.

\begin{df}
An algebra $A$ is said to be $\Gamma$-graded if it is $\Gamma$-graded as vector space and $|a\cdot b|=|a|+|b|$ for all homogeneous $a,b$ in $A$.
\end{df}

\begin{ex} \label{rem grad somme et tens} Let $V$ and $W$ be finite-dimensional $\Gamma$-graded vector spaces.
\begin{enumerate}[label=\alph*)]
\item The base field $k$ has a trivial $\Gamma$-gradation given by $|a|=0$ for all $a$ in $k$.
\item The vector space $Hom(V,W)$ is $\Gamma$-graded by $Hom(V,W)=\bigoplus \limits_{\gamma \in \Gamma} Hom(V,W)_{\gamma}$ where
 $$Hom(V,W)_{\gamma}:=\lbrace f\in Hom(V,W) ~ | ~ f(V_a)\subseteq W_{a+\gamma} \quad \forall a \in \Gamma \rbrace.$$
\item The vector space $V\oplus W$ is $\Gamma$-graded by $V\oplus W=\bigoplus \limits_{\gamma \in \Gamma} (V\oplus W)_{\gamma}$ where
$$(V\oplus W)_{\gamma}=V_{\gamma}\oplus W_{\gamma}.$$
\item The vector space $V\otimes W$ is $\Gamma$-graded by $V\otimes W=\bigoplus \limits_{\gamma \in \Gamma} (V\otimes W)_{\gamma}$ where
$$(V\otimes W)_{\gamma}=\bigoplus \limits_{a+b=\gamma} V_{a}\otimes W_{b}.$$
\item As we have seen, the vector space $End(V)=Hom(V,V)$ is $\Gamma$-graded. In fact, the associative algebra $End(V)$ is also $\Gamma$-graded as an algebra.
\item Let $A$ and $B$ be $\Gamma$-graded algebras. As we have seen, the vector space $A\oplus B$ is $\Gamma$-graded. In fact, the algebra $A\oplus B$ is also $\Gamma$-graded as an algebra.
\end{enumerate}
\end{ex}

\subsection{Commutation factors and representations of the symmetric group} \label{Section CLA bicharacters}

Let $\Gamma$ be an abelian group. In this subsection we introduce the notion of a commutation factor of $\Gamma$. This allows us to define a notion of ``commutative'' and ``anticommutative'' for $\Gamma$-graded algebras which takes into account the $\Gamma$-gradation. It also allows us to modify the standard actions of the symmetric group $S_n$ on the $n$-fold tensor product of a $\Gamma$-graded vector space.

\begin{df} (See III.116 in \cite{Bourbaki70}) Let $\Gamma$ be an abelian group. A commutation factor $\epsilon$ of $\Gamma$ is a map $\epsilon : \Gamma \times \Gamma \rightarrow k^*$ such that for all $a,b,c \in \Gamma$
\begin{align*}
\epsilon(a,b)\epsilon(b,a)&=1,\\
\epsilon(a+b,c)&=\epsilon(a,c)\epsilon(b,c),\\
\epsilon(a,b+c)&=\epsilon(a,b)\epsilon(a,c).
\end{align*}
\end{df}

The basic features of commutation factors are given in the following remark (\cite{Scheunert79}).

\begin{rem} \label{decolouration 1}
\begin{enumerate}[label=\alph*)]
\item We have
$$\epsilon(a,0)=\epsilon(0,a)=1, \qquad \epsilon(a,-b)=\epsilon(b,a) \qquad \forall a,b \in \Gamma.$$
\item For $a \in \Gamma$, we have $\epsilon(a,a)=\pm 1$ and hence a partition $\Gamma=\Gamma_0 \cup \Gamma_1$ where
$$\Gamma_0:=\lbrace a \in \Gamma ~ | ~ \epsilon(a,a)=1 \rbrace, \qquad \Gamma_1:=\lbrace a \in \Gamma ~ | ~ \epsilon(a,a)=-1 \rbrace.$$
The map $a \mapsto \epsilon(a,a)\in \mathbb{Z}_2$ is a group homomorphism so $\Gamma_0$ is a normal subgroup of index at most two.
\end{enumerate}
\end{rem}

Here are some examples of non-trivial commutation factors.

\begin{ex} \label{first examples antisym bichar}
\begin{enumerate}[label=\alph*)]
\item The most important non-trivial example of a commutation factor is obtained by taking $\Gamma=\mathbb{Z}_2$ and $\epsilon$ defined by 
$$\epsilon(a,b):=(-1)^{ab} \qquad \forall a,b\in \mathbb{Z}_2.$$
\item Let $\Gamma$ be an abelian group together with commutation factor $\epsilon$. Then, $\tilde{\epsilon} : (\mathbb{Z}\times \Gamma) \times (\mathbb{Z}\times \Gamma) \rightarrow k^*$ given by
$$\tilde{\epsilon}((m,\gamma),(m',\gamma')):=(-1)^{mm'}\epsilon(\gamma,\gamma') \quad \forall m,m' \in \mathbb{Z}, ~ \forall \gamma,\gamma' \in \Gamma$$
is a commutation factor of $\mathbb{Z}\times \Gamma$ and the same formula defines a commutation factor of $\mathbb{Z}_2\times \Gamma$.
\end{enumerate}
\end{ex}

In \cite{Scheunert79}, Scheunert gives the general form of a commutation factor of a finitely generated abelian group in terms of a cyclic decomposition.

\begin{notation} Let $V$ be a $\Gamma$-graded vector space, let $\epsilon$ be a commutation factor of $\Gamma$ and let $v,w\in V$. For brevity, we denote $\epsilon(|v|,|w|)$ by $\epsilon(v,w)$ and by $\mathcal{E}$ the canonical linear map $\mathcal{E} : V\rightarrow V$ given by
$$\mathcal{E}(v):=\epsilon (v,v)v \qquad \forall v\in V.$$
\end{notation}

It-is known that if $V$ is a $\Gamma$-graded vector space, then using a commutation factor one can modify the standard actions of the symmetric group $S_n$ on $V^{\otimes n}$ to take into account the $\Gamma$-grading of $V$. We will need the following result.

\begin{pp} (See \cite{Scheunert83}) \label{pp action}
Let $V$ be a $\Gamma$-graded vector space and let $\epsilon$ be a commutation factor of $\Gamma$. There is a unique right group action $\pi : S_n \rightarrow GL(V^{\otimes n})$ of the permutation group $S_n$ on $V^{\otimes n}$ such that the action of a transposition $\tau_{i,i+1} \in S_n$ is given by
\begin{equation*}
\pi(\tau_{i,i+1})(v_1\otimes \ldots \otimes v_n)=-\epsilon(v_i,v_{i+1}) v_1 \otimes \ldots \otimes v_{i+1}\otimes v_i \otimes \ldots \otimes v_n
\end{equation*}
for all $v_1,\ldots,v_n \in V$. For an arbitrary element $\sigma \in S_n$, this action is given by
$$\pi(\sigma)(v_1\otimes \ldots \otimes v_n)=p(\sigma ; v_1,\ldots, v_n) v_{\sigma(1)}\otimes \ldots \otimes v_{\sigma(n)}$$ 
where
$$p(\sigma ; v_1,\ldots, v_n)=sgn(\sigma) \displaystyle \prod_{\substack{1 \leq i < j \leq n \\ \sigma^{-1}(i)>\sigma^{-1}(j) }} \epsilon(v_i,v_j).$$
\end{pp}
\vspace{0.1cm}

\begin{rem} \label{rem formule de p}
\begin{enumerate}[label=\alph*)]
\item For $v_1,\ldots,v_n \in V$ and $\sigma,\sigma' \in S_n$, we have
\begin{align*}
p(\sigma\sigma' ; v_1,\ldots, v_n)&=p(\sigma' ; v_{\sigma(1)},\ldots, v_{\sigma(n)})p(\sigma ; v_1,\ldots, v_n),\\
p(Id ; v_1,\ldots,v_n)&=1.
\end{align*}
This shows that $p$ is a right multiplier in the sense of \cite{Bruhat56}.
\item Let $n,m\in\mathbb{N}$ such that $n\leq m$ and let $i\in \llbracket 0,m-n \rrbracket$. There is a group inclusion $S_n \subseteq S_m$ where $S_n$ acts only on the coordinates $\llbracket i+1,i+n \rrbracket$ and furthermore (with the obvious notations)
$$p_n(\sigma ; v_{i+1},\ldots ,v_{i+n})=p_{m}(\sigma ; v_1,\ldots ,v_{m}) \qquad \forall \sigma \in S_n, ~ \forall v_1,\ldots,v_{m}\in V.$$
\end{enumerate}
\end{rem}

\subsection{$\epsilon$-symmetric and $\epsilon$-antisymmetric bilinear maps} \label{section symmetric bilinear forms}

Let $\Gamma$ be an abelian group and let $\epsilon$ be a commutation factor of $\Gamma$. In this subsection we define $\epsilon$-symmetric and $\epsilon$-antisymmetric bilinear maps for $\Gamma$-graded vector spaces.

\begin{df}
Let $V$ and $W$ be $\Gamma$-graded vector spaces and let $B : V\times V \rightarrow W$ be a bilinear map.
\begin{enumerate}[label=\alph*)]
\item We say that $B$ is $\epsilon$-symmetric if $B(v,w)=\epsilon(v,w)B(w,v)$ for all $v,w\in V$.
\item We say that $B$ is $\epsilon$-antisymmetric if $B(v,w)=-\epsilon(v,w)B(w,v)$ for all $v,w\in V$.
\end{enumerate}
\end{df}

\begin{rem}
If $W=k$, unless otherwise stated, we always assume that $B$ is of degree $0$.
\end{rem}

It turns out that an $\epsilon$-symmetric bilinear form is the orthogonal sum of a symmetric bilinear form and an antisymmetric bilinear form in the following sense.

\begin{pp} \label{pp V=V0+V1} (see \cite{ChenKang16})
Let $V$ be a $\Gamma$-graded vector space and let $(\phantom{v},\phantom{v}) : V\times V \rightarrow k$ be an $\epsilon$-symmetric bilinear form. Let $V=V_0\oplus V_1$ where
$$V_0=\lbrace v \in V ~ | ~ \epsilon(v,v)=1 \rbrace, \qquad V_1=\lbrace v \in V ~ | ~ \epsilon(v,v)=-1 \rbrace.$$
Then, the restriction of $(\phantom{v},\phantom{v})$ to $V_0$ is symmetric in the usual sense, the restriction of $(\phantom{v},\phantom{v})$ to $V_1$ is antisymmetric in the usual sense, and $V_0$ is orthogonal to $V_1$.
\end{pp}

\subsection{$\epsilon$-alternating multilinear maps}

Let $\Gamma$ be an abelian group and let $\epsilon$ be a commutation factor of $\Gamma$. In this subsection we define $\epsilon$-alternating maps for $\Gamma$-graded vector spaces and the product and composition of $\epsilon$-alternating maps.
\vspace{0.1cm}

Let $V$ and $W$ be $\Gamma$-graded vector spaces. Let $\tilde{\pi} : S_n \rightarrow GL(Hom(V^{\otimes n},W))$ be the left group action given by
$$(\tilde{\pi}(\sigma)(f))(\bar{v}):=f(\pi(\sigma)(\bar{v})) \qquad \forall \sigma \in S_n, ~ \forall f \in Hom(V^{\otimes n},W), ~ \forall \bar{v} \in V^{\otimes n}$$
where $\pi : S_n \rightarrow GL(V^{\otimes n})$ is the right group action given by Proposition \ref{pp action}.

\begin{df}
Let $V$ and $W$ be finite-dimensional $\Gamma$-graded vector spaces. We define the $\Gamma$-graded vector space $Alt^i_{\epsilon}(V,W)$ by
$$Alt^i_{\epsilon}(V,W):=\lbrace f \in Hom(V^{\otimes i},W) ~ | ~ \tilde{\pi}(\sigma)(f)=f \quad \forall \sigma \in S_n \rbrace.$$
We also define the $\mathbb{Z}\times\Gamma$-graded vector space
$$Alt_{\epsilon}(V,W):=\bigoplus \limits_{i\in \mathbb{N}} Alt^i_{\epsilon}(V,W)$$
and denote $Alt_{\epsilon}(V,k)$ by $Alt_{\epsilon}(V)$.
\end{df}

\begin{rem} \label{rem def equivalente alternating map}
Each $f\in Alt^i_{\epsilon}(V,W)$ is uniquely determined by the multilinear map $\tilde{f} : V^i \rightarrow W$ given by $\tilde{f}(v_1,\ldots,v_i)=f(v_1\otimes \ldots \otimes v_i)$ where $v_1,\ldots,v_i \in V$. This map $\tilde{f}$ satisfies
$$\tilde{f}(v_1,\ldots,v_{l},v_{l+1},\ldots,v_i)=-\epsilon(v_l,v_{l+1})\tilde{f}(v_1,\ldots,v_{l+1},v_l,\ldots,v_i)$$
for all $v_1,\ldots,v_i \in V$ and $l\in \llbracket 1,i-1 \rrbracket$. Conversely given a multilinear map $g : V^i\rightarrow W$ with this property there is a unique element $f\in Alt^i_{\epsilon}(V,W)$ such that $\tilde{f}=g$.
\end{rem}

\begin{df}
Let $I:=\llbracket 1,n\rrbracket$ and let $I_1,\ldots,I_m$ be disjoint subsets of $I$ such that $\bigcup \limits_{i\in \llbracket 1,m \rrbracket} I_i=I$. We denote by $S(I_1,\ldots,I_m)$ the set of all permutations $\sigma \in S_n$ which satisfy
$$\forall i \in \llbracket 1,m\rrbracket, ~ a,b\in I_i \text{ and } a<b \quad \Rightarrow \quad \sigma(a)<\sigma(b).$$
Such a permutation is called a shuffle permutation.
\end{df}

\begin{rem} \label{rem shuffle}
\begin{enumerate}[label=\alph*)]
\item The cardinal of $S(I_1,\ldots,I_m)$ is given by
$$|S(I_1,\ldots,I_m)|=\dbinom{n}{|I_1|}\dbinom{n-|I_1|}{|I_2|}\ldots \dbinom{n-|I_1|-\ldots -|I_{m-2}|}{|I_{m-1}|}.$$
\item There are group inclusions $S_{|I_i|} \subseteq S_{n}$, for all $i \in \llbracket 1,m \rrbracket$, where $S_{|I_i|}$ acts only on the coordinates
$$\left\{  \sum \limits_{k\in \llbracket 1,i-1 \rrbracket} |I_k|+1 , \sum \limits_{k\in \llbracket 1,i-1 \rrbracket} |I_k| +2,\ldots,  \sum \limits_{k\in \llbracket 1,i \rrbracket} |I_k| \right\} .$$
Furthermore, for each $\sigma \in S_n$, there exist unique $\sigma' \in S(I_1,\ldots,I_m)$ and $\sigma_i \in S_{|I_i|} ~ \forall i \in \llbracket 1,m \rrbracket$ such that
$$\sigma=\sigma'\circ \sigma_1 \circ \ldots \circ \sigma_m.$$
\end{enumerate}
\end{rem}
\vspace{0.1cm}

Let $T,U,V$ and $W$ be finite-dimensional $\Gamma$-graded vector spaces, let $f \in Alt^i_{\epsilon}(T,U)$, let $g \in Alt^j_{\epsilon}(T,V)$ and let $\phi : U\times V \rightarrow W$ be a bilinear map. Let $(fg)_{\phi} : T^{\otimes (i+j)} \rightarrow W$ be given by
$$(fg)_{\phi}(v_1\otimes\ldots\otimes v_{i+j}):=\phi(f(v_1\otimes\ldots\otimes v_i),g(v_{i+1}\otimes\ldots\otimes v_{i+j})) \quad \forall v_1,\ldots,v_{i+j}\in T.$$
The exterior product of $f$ and $g$ is now defined by ``antisymmetrising'' this product only with respect to shuffle permutations.

\begin{df} \label{df exterior product}
With the notation above, the map $f\wedge_{\phi} g : T^{\otimes (i+j)} \rightarrow W$ is defined by
$$f\wedge_{\phi} g:=\sum \limits_{\sigma \in S(\llbracket1,i\rrbracket,\llbracket i+1,i+j \rrbracket)} \tilde{\pi}(\sigma)((fg)_{\phi}).$$
\end{df}

\begin{rem}
The relation between this exterior product and ``antisymmetrisation'' over all permutations is:
$$\sum \limits_{\sigma \in S_{i+j}} \tilde{\pi}(\sigma)((fg)_{\phi})=i!j!f\wedge_{\phi} g.$$
In characteristic zero the two possible definitions of an ``exterior product'' are equivalent.
\end{rem}

\begin{pp} \label{pp product of mult forms}
The map $f\wedge_{\phi} g$ is in $Alt^{i+j}_{\epsilon}(T,W).$
\end{pp}

\begin{proof}
This proof is based on the proof of the analogous result for classical exterior forms in \cite{Cartan67}. Here, we use implicitly the formulae of Remark \ref{rem formule de p}.
\vspace{0.2cm}

To prove the proposition it is sufficient to show that if $l\in \llbracket1,i+j-1\rrbracket$,
$$f\wedge_{\phi} g(v_1\otimes\ldots\otimes v_l\otimes v_{l+1}\otimes \ldots\otimes v_{i+j})=-\epsilon(v_l,v_{l+1})f\wedge_{\phi} g(v_1\otimes\ldots\otimes v_{l+1}\otimes v_l\otimes \ldots\otimes v_{i+j})$$
for all $v_1,\ldots,v_{i+j}$ in $T$. Consider the right-hand side of this equation. We have
\begin{align*}
&-\epsilon(v_l,v_{l+1})f\wedge_{\phi} g(v_1\otimes\ldots\otimes v_{l+1}\otimes v_l\otimes\ldots\otimes v_{i+j})\\
&=-\epsilon(v_l,v_{l+1})\sum \limits_{\sigma \in S(\llbracket1,i\rrbracket,\llbracket i+1,i+j \rrbracket)} (fg)_{\phi}\left(\pi(\sigma)(v_1\otimes \ldots \otimes v_{l+1}\otimes v_l \otimes \ldots \otimes v_{i+j})\right).
\end{align*}
We divide the permutations $\sigma \in S(\llbracket1,i\rrbracket,\llbracket i+1,i+j \rrbracket)$ in two categories:
\vspace{0.2cm}

1. Those $\sigma$ for which $\sigma^{-1}(l)$ and $\sigma^{-1}(l+1)$ are both integers $\leq i$ or both $\geq i+1$. In the first case, $v_l$ and $v_{l+1}$ occur amongst the first i places in $f(v_{\sigma(1)}\otimes \ldots\otimes v_{\sigma(i)})$ and hence, since $f$ is $\epsilon$-alternating, we have
$$-\epsilon(v_l,v_{l+1}) (fg)_{\phi}\left(\pi(\sigma)(v_1\otimes \ldots \otimes v_{l+1}\otimes v_l \otimes \ldots \otimes v_{i+j})\right)=(fg)_{\phi}\left(\pi(\sigma)(v_1\otimes \ldots \otimes v_{i+j})\right).$$
In the second case, we have the same relation since $g$ is $\epsilon$-alternating.
\vspace{0.2cm}

2. The second category is itself divided into two sub-categories: those $\sigma$ for which $\sigma^{-1}(l)\leq i$ and $\sigma^{-1}(l+1)\geq i+1$ and those $\sigma$ for which $\sigma^{-1}(l)\geq i$ and $\sigma^{-1}(l+1)\leq i+1$. Let $\tau$ be the transposition which interchanges $l$ and $l+1$. If $\sigma$ is in the first sub-category, $\tau\sigma$ is in the second, and vice versa. We may therefore group in pairs the remaining terms as follows: for each $\sigma$ such that $\sigma^{-1}(l)\leq i$ and $\sigma^{-1}(l+1)\geq i+1$  we have
\begin{align*}
&-\epsilon(v_l,v_{l+1}) \Big( (fg)_{\phi}\left(\pi(\sigma)(v_1\otimes \ldots \otimes v_{l+1}\otimes v_l \otimes \ldots \otimes v_{i+j})\right) \\
&+ (fg)_{\phi}\left(\pi(\tau\sigma)(v_1\otimes \ldots \otimes v_{l+1}\otimes v_l \otimes \ldots \otimes v_{i+j})\right) \Big) \\
&= (fg)_{\phi}\left(\pi(\sigma)(v_1\otimes \ldots \otimes v_{i+j})\right) + (fg)_{\phi}\left(\pi(\tau\sigma)(v_1\otimes \ldots \otimes v_{i+j})\right).
\end{align*}
\end{proof}

The most important example of the above construction is when $U=V=W=k$ and $\phi : k\times k \rightarrow k$ is the product. In this case we denote $\wedge_{\phi}$ by $\wedge_{\epsilon}$ and this defines a product on $Alt_{\epsilon}(V)$. One can show that:

\begin{pp}
With respect to $\wedge_{\epsilon}$, the algebra $Alt_{\epsilon}(V)$ is $\mathbb{Z}\times\Gamma$-graded, $\tilde{\epsilon}$-commutative (see Example \ref{first examples antisym bichar}) and associative.
\end{pp}

Later on, we will need the following definition.

\begin{df} \label{def norm}
Let $T,U,W$ be finite-dimensional $\Gamma$-graded vector spaces and $\phi : U\times U \rightarrow W$ a bilinear map. The norm $N_{\phi}(f)$ of $f \in Alt_{\epsilon}(T,U)$ is defined by
$$N_{\phi}(f):=f\wedge_{\phi} f.$$
This is an element of $Alt^{2i}_{\epsilon}(T,W)$ if $f\in Alt^i_{\epsilon}(T,U)$.
\end{df}
\vspace{0.2cm}

Finally to complete this subsection, we define a composition of $\epsilon$-alternating multilinear maps. Let $U,V$ and $W$ be finite-dimensional $\Gamma$-graded vector spaces, let $f \in Alt^i_{\epsilon}(U,V)$ and let $g \in Alt^j_{\epsilon}(W,U)$. We first define $f*g : W^{\otimes(ij)} \rightarrow V$ by
$$f*g(v_1,\ldots,v_{ij}):=f(g(v_1\otimes\ldots\otimes v_j)\otimes g(v_{j+1}\otimes\ldots\otimes v_{2j})\otimes\ldots\otimes g(v_{(i-1)j+1}\otimes\ldots\otimes v_{ij}))$$
for all $v_1\otimes\ldots\otimes v_{ij} \in W^{\otimes (ij)}$.
\vspace{0.1cm}

The exterior composition of $f$ and $g$ is now defined by ``antisymmetrising'' this composition only with respect to certain shuffle permutations.

\begin{df} \label{df exterior composition}
With the notation above, the map $f\circ g : W^{\otimes(ij)} \rightarrow V$ is defined by
$$f\circ g:=\sum \limits_{\sigma \in S(\llbracket1,j\rrbracket,\ldots,\llbracket (i-1)j+1,ij \rrbracket)} \tilde{\pi}(\sigma)(f*g).$$
\end{df}

\begin{rem}
The relation between this and antisymmetrising over all permutations is
$$\sum \limits_{\sigma \in S_{ij}} \tilde{\pi}(\sigma)(f*g)=(j!)^i f\circ g.$$
\end{rem}

\begin{pp}
The map $f\circ g$ is in to $Alt^{ij}_{\epsilon}(W,V).$
\end{pp}

\begin{proof}
The proof is similar to the proof of Proposition \ref{pp product of mult forms}.
\end{proof}

\subsection{Colour Lie algebras} \label{Section CLA}

Let $\Gamma$ be an abelian group and let $\epsilon$ be a commutation factor of $\Gamma$. In this subsection following \cite{Ree60}, \cite{RittenbergWyler78-2}, \cite{RittenbergWyler78-1}, \cite{Scheunert79} we define colour Lie algebras.

\begin{df} \label{def CLA}
A colour Lie algebra is a $\Gamma$-graded vector space $\gg=\bigoplus \limits_{\gamma \in \Gamma} \gg_{\gamma}$ together with a bilinear map $\lbrace \phantom{x}, \phantom{x} \rbrace : \gg \times \gg \rightarrow \gg $ such that 
\begin{enumerate}[label=\alph*)]
\item $\lbrace \gg_{\alpha}, \gg_{\beta} \rbrace \subseteq \gg_{\alpha+\beta}$ for all $\alpha,\beta \in \Gamma$,
\item $\lbrace x,y \rbrace=-\epsilon(x,y)\lbrace y,x \rbrace$ for all $x,y \in \gg \qquad$ ($\epsilon$-antisymmetry),
\item $\epsilon(z,x)\lbrace x , \lbrace y , z \rbrace \rbrace+\epsilon(x,y)\lbrace y , \lbrace z , x \rbrace \rbrace+\epsilon(y,z)\lbrace z , \lbrace x , y \rbrace \rbrace=0$ for all $x,y,z \in \gg \qquad$

($\epsilon$-Jacobi identity).
\end{enumerate}
\end{df}

Lie (super)algebras are colour Lie algebras, here are some examples:

\begin{ex}
Let $V$ be a finite-dimensional $\Gamma$-graded vector space.
\begin{enumerate}[label=\alph*)]
\item The associative $\Gamma$-graded algebra $End(V)$ is a colour Lie algebra for the bracket $\lbrace a,b \rbrace := ab-\epsilon(a,b)ba$ for all $a,b$ in $End(V)$ and is denoted $\gl_{\epsilon}(V)$.
\item Let $(\phantom{v},\phantom{v}) : V \times V \rightarrow k$ be an $\epsilon$-symmetric bilinear form. We set
$$\so_{\epsilon}(V,(\phantom{v},\phantom{v})):=\bigoplus \limits_{\gamma \in \Gamma} \so_{\epsilon}(V,(\phantom{v},\phantom{v}))_{\gamma}$$
where
$$\so_{\epsilon}(V,(\phantom{v},\phantom{v}))_{\gamma}:=\lbrace f \in End(V)_{\gamma} ~ | ~ (f(v),w)+\epsilon(f,v)(v,f(w))=0 \quad  \forall v,w \in V \rbrace.$$
One can show that $\so_{\epsilon}(V,(\phantom{v},\phantom{v}))$ is stable under the bracket of $\gl_{\epsilon}(V)$ and hence is a colour Lie algebra.
\end{enumerate}
\end{ex}

We now define morphisms and representations of colour Lie algebras.

\begin{df}
\begin{enumerate}[label=\alph*)]
\item Let $\gg$ and $\gg'$ be colour Lie algebras. A degree $0$ linear map $f \in Hom(\gg, \gg')$ is a morphism of colour Lie algebras if $f(\lbrace x,y \rbrace)=\lbrace f(x),f(y) \rbrace$ for all $x,y \in \gg$. Furthermore, we say that $\gg$ and $\gg'$ are isomorphic colour Lie algebras if $f$ is a linear isomorphism.
\item Let $\gg$ be a colour Lie algebra. A finite-dimensional representation $V$ of $\gg$ is a finite-dimensional $\Gamma$-graded vector space $V$ together with a morphism of colour Lie algebras $\rho : \gg \rightarrow \gl_{\epsilon}(V)$. We sometimes write $x(v)$ instead of $\rho(x)(v)$ for $x \in \gg$ and $v \in V$.
\item Let $\gg$ be a colour Lie algebra. A finite-dimensional $\epsilon$-orthogonal representation $V$ of $\gg$ is a finite-dimensional $\Gamma$-graded vector space $V$ together with a non-degenerate $\epsilon$-symmetric bilinear form $(\phantom{v},\phantom{v})$ and a morphism of colour Lie algebras $\rho : \gg \rightarrow \so_{\epsilon}(V,(\phantom{v},\phantom{v}))$.
\end{enumerate}
\end{df}

\begin{ex} \label{ex rep osp o et sp}
Orthosymplectic representations of Lie superalgebras are examples of $\epsilon$-orthogonal representations, where $\Gamma=\mathbb{Z}_2$ and $\epsilon(a,b)=(-1)^{ab}$ for all $a,b \in \mathbb{Z}_2$. In particular both orthogonal and symplectic representations of Lie algebras are examples of $\epsilon$-orthogonal representations.
\end{ex}

\subsection{$\epsilon$-symmetric bilinear forms on colour Lie algebras} \label{Subsection symmetric bilinear form on CLA}

Let $\Gamma$ be an abelian group and let $\epsilon$ be a commutation factor of $\Gamma$. In this subsection we define invariant $\epsilon$-symmetric bilinear forms for colour Lie algebras. We show that the ``$\epsilon$-trace'' in the fundamental representation defines an invariant, $\epsilon$-symmetric and non-degenerate bilinear form on $\gl_{\epsilon}$ and $\so_{\epsilon}$.

\begin{df} \label{def quadratic}
\begin{enumerate}[label=\alph*)]
\item Let $\gg$ be a colour Lie algebra. A bilinear form $B : \gg \times \gg \rightarrow k$ is $ad$-invariant if for all $x,y,z \in \gg$,
$$B(\lbrace x,y \rbrace , z)=-\epsilon(x,y)B(y,\lbrace x,z \rbrace).$$
\item An $\epsilon$-quadratic colour Lie algebra is a colour Lie algebra together with a bilinear form which is $\epsilon$-symmetric, $ad$-invariant and non-degenerate.
\item (see \cite{Scheunert83})
Let $V$ be a finite-dimensional $\Gamma$-graded vector space and let $f\in End(V)$. Recall that $\mathcal{E} : V \rightarrow V$ is defined by $\mathcal{E}(v)=\epsilon(v,v)v$ for all $v\in V$. The $\epsilon$-trace of $f$ is defined by:
$$Tr_{\epsilon}(f):=Tr(\mathcal{E}\circ f).$$
\end{enumerate}
\end{df}

\begin{pp} \label{gl and so are quadratic}
Let $V$ be a finite-dimensional $\Gamma$-graded vector space.
\begin{enumerate}[label=\alph*)]
\item \label{gl is quadratic} The bilinear form $\gl_{\epsilon}(V)\times \gl_{\epsilon}(V)\rightarrow k$ given by
\begin{equation} \label{bilinear form on gl}
Tr_{\epsilon}(fg) \qquad \forall f,g \in \gl_{\epsilon}(V)
\end{equation}
is $ad$-invariant, $\epsilon$-symmetric and non-degenerate.
\item \label{so is quadratic} Let $(\phantom{v},\phantom{v})$ be a non-degenerate $\epsilon$-symmetric bilinear form on $V$ and suppose that $dim(V)\geq 2$. Then the restriction of \eqref{bilinear form on gl} to $\so_{\epsilon}(V,(\phantom{v},\phantom{v}))$ is non-degenerate.
\end{enumerate}
\end{pp}

\begin{proof}
\begin{enumerate}[label=\alph*)]
\item The bilinear form of Equation \eqref{bilinear form on gl} is $ad$-invariant and $\epsilon$-symmetric by a straightforward calculation. It is non-degenerate since in a $\Gamma$-homogeneous basis, the elementary matrices $E_{v_{\gamma},w_{\gamma'}},E_{w_{\gamma'},v_{\gamma}}$ corresponding to homogeneous vectors $v_{\gamma}\in V_{\gamma},w_{\gamma'}\in V_{\gamma'}$ satisfy
$$Tr_{\epsilon}(E_{v_{\gamma},w_{\gamma'}} E_{w_{\gamma'},v_{\gamma}})\neq 0.$$
\item By \ref{gl is quadratic} the bilinear form considered is $\epsilon$-symmetric and $ad$-invariant. We need the following lemma.

\begin{lm} \label{lm canonical moment map}
Let $V$ be a finite-dimensional $\Gamma$-graded vector space such that $dim(V)\geq 2$ together with a non-degenerate $\epsilon$-symmetric bilinear form $(\phantom{v},\phantom{v})$. Let $u,v\in V$. The map $f(u,v) \in \so_{\epsilon}(V,(\phantom{v},\phantom{v}))$ defined by
$$f(u,v)(w):=\epsilon(v,w)(u,w)v-(v,w)u \qquad \forall w\in V$$
satisfies
$$Tr_{\epsilon}(f\circ f(u,v))=-2(f(u),v) \qquad \forall f\in \so_{\epsilon}(V,(\phantom{v},\phantom{v})).$$
\end{lm}

\begin{proof}
Let $u,v \in V$. One can check that $f(u,v) \in \so_{\epsilon}(V,(\phantom{v},\phantom{v}))$. Let $\lbrace e_i :  1\leq i \leq dim(V) \rbrace$ be a homogeneous basis of $V$. Let $e_i,e_j\in \lbrace e_i : 1\leq i \leq dim(V) \rbrace$ and $f \in \so_{\epsilon}(V,(\phantom{v},\phantom{v}))$. Let $\lbrace e^i :  1\leq i \leq dim(V) \rbrace$ be its dual basis in the sense that $(e_i,e^j)=\delta_{ij}$ for all $i,j \in \llbracket 1,dim(V) \rrbracket$. We have $|e^i|=-|e_i|$ and
$$f(e_i,e_j)(e^k)=\epsilon(e_k,e_j)\delta_{ik}e_j-\delta_{jk}e_i.$$
After a straightforward calculation we have
\begin{equation} \label{preuve formule mucan 1}
\frac{(\mathcal{E}(f(\mu(e_i,e_j)(e^k))),e_k)}{\epsilon(e_k,e_k)}=-\delta_{ik}(f(e_k),e_j)-\delta_{jk}(f(e_i),e_k).
\end{equation}
Since
$$\sum\limits_{k=1}^{dim(V)} \frac{(f(e^k),e_k)}{\epsilon(e_k,e_k)}=Tr(f) \qquad \forall f\in \gl_{\epsilon}(V)$$
by Equation \eqref{preuve formule mucan 1} we obtain
$$Tr_{\epsilon}( f\circ f(e_i,e_j))=-2(f(e_i),e_j).$$
\end{proof}

Let $f\in \so_{\epsilon}(V,(\phantom{v},\phantom{v}))$ be such that
$$Tr_{\epsilon}( f\circ g)=0\qquad \forall g \in \so_{\epsilon}(V,(\phantom{v},\phantom{v})).$$
In particular, we have
$$Tr_{\epsilon}( f\circ f(u,v))=0\qquad \forall u,v \in V.$$
By Lemma \ref{lm canonical moment map}, we have
$$(f(u),v) \qquad \forall u,v\in V$$
then, since $(\phantom{v},\phantom{v})$ is non-degenerate, we obtain that $f\equiv 0$ and so the restriction of \eqref{bilinear form on gl} to $\so_{\epsilon}(V,(\phantom{v},\phantom{v}))$ is non-degenerate.
\end{enumerate}
\end{proof}

\section{Moment map of $\epsilon$-orthogonal representations} \label{section moment map}

Let $\Gamma$ be an abelian group and let $\epsilon$ be a commutation factor of $\Gamma$. In this section, we define the moment map of an $\epsilon$-orthogonal representation $(V,(\phantom{v},\phantom{v}))$ of an $\epsilon$-quadratic colour Lie algebra $(\gg,B_{\gg})$. After giving some general properties, we study the moment map of the fundamental representation of $\so_{\epsilon}(V,(\phantom{v},\phantom{v}))$ which we will call the ``canonical'' moment map. Finally, we give a formula for the moment map of a tensor product of $\epsilon$-orthogonal representations in terms of the moment maps of the factors.
\vspace{0.2cm}

Unless otherwise stated, we suppose all $\epsilon$-orthogonal representations of dimension at least two.

\begin{df} \label{df moment map}
Let $\rho : \gg \rightarrow \so_{\epsilon}(V,(\phantom{v},\phantom{v}))$ be a finite-dimensional $\epsilon$-orthogonal representation of a finite-dimensional $\epsilon$-quadratic colour Lie algebra $(\gg,B_{\gg})$. We define the moment map of the representation $\rho : \gg \rightarrow \so_{\epsilon}(V,(\phantom{v},\phantom{v}))$ to be the bilinear map $\mu : V\times V \rightarrow \gg$ given by
$$B_{\gg}(x,\mu(v,w))=(\rho(x)(v),w) \qquad \forall v,w \in V, ~ \forall x \in \gg.$$
\end{df}
\vspace{0.1cm}

This generalises the usual moment map of a symplectic representation of a quadratic Lie algebra. We now show that the moment map is $\epsilon$-antisymmetric and equivariant.
\vspace{0.1cm}

\begin{pp} \label{proprietes de mu} Let $\rho : \gg \rightarrow \so_{\epsilon}(V,(\phantom{v},\phantom{v}))$ be a finite-dimensional $\epsilon$-orthogonal representation of a finite-dimensional $\epsilon$-quadratic colour Lie algebra $(\gg,B_{\gg})$ and let $\mu$ be its moment map. Then:
\begin{enumerate}[label=\alph*)]
\item The map $\mu$ is of degree $0$.
\item We have $\mu \in Alt_{\epsilon}^2(V,\gg)$.
\item For $x \in \gg$ and $v,w \in V$, we have
\begin{equation*}
\lbrace x , \mu (v,w) \rbrace=\mu(x(v),w)+\epsilon(x,v)\mu(v,x(w)).
\end{equation*}
\item Let $\lbrace e_i  :  1\leq i \leq dim(\gg) \rbrace$ be a basis of $\gg$ and let $\lbrace e^i  :  1\leq i \leq dim(\gg) \rbrace$ be the dual basis in the sense that $B_{\gg}(e_i,e^j)=\delta_{ij}$. We have
\begin{equation*}
\mu(v,w)=\sum\limits_{i=1}^{dim(\gg)} (e_i(v),w)e^i \qquad \forall v,w \in V.
\end{equation*}
\end{enumerate}
\end{pp}

\begin{proof}
The map $\mu$ is of degree $0$ since $B_{\gg}$ and $(\phantom{v},\phantom{v})$ are of degree $0$. This proves $a)$, we now give a proof of $b)$ and the other parts of the proposition can be proved similarly.
\vspace{0.1cm}

For $x \in \gg$ we have
\begin{align*}
B_{\gg}(x,\mu(v,w))&=(\rho(x)(v),w)\\
&=-\epsilon(x,v)(v,\rho(x)(w))\\
&=-\epsilon(x,v)\epsilon(v,x+w)(\rho(x)(w),v)\\
&=-\epsilon(v,w)B_{\gg}(x,\mu(w,v)).
\end{align*}
\end{proof}

\begin{rem} \label{modification de mu par une cst}
Let $\rho : \gg \rightarrow \so_{\epsilon}(V,(\phantom{v},\phantom{v}))$ be a finite-dimensional $\epsilon$-orthogonal representation of a finite-dimensional $\epsilon$-quadratic colour Lie algebra $(\gg,B_{\gg})$ and let $\mu$ be its moment map. Let $\alpha, \beta \in k^*$. Then $\gg$ is also $\epsilon$-quadratic for the bilinear form $\alpha \cdot B_{\gg}$ and $\rho : \gg \rightarrow \so_{\epsilon}(V,\beta \cdot (\phantom{v},\phantom{v}))$ is also an $\epsilon$-orthogonal representation of $\gg$. The corresponding moment map $\mu_{\alpha,\beta}$ satisfies
$$\mu_{\alpha,\beta}(v,w)=\frac{\beta}{\alpha}\mu(v,w) \qquad \forall v,w\in V.$$
\end{rem}
\vspace{0.1cm}

We now study the moment map of the fundamental representation of $\so_{\epsilon}(V,(\phantom{v},\phantom{v}))$ which we call the ``canonical'' moment map.
\vspace{0.1cm}

\begin{pp} \label{pp moment map canonique}
Let $V$ be a finite-dimensional $\Gamma$-graded vector space together with a non-degenerate $\epsilon$-symmetric bilinear form $(\phantom{v},\phantom{v})$. Consider the $\epsilon$-orthogonal representation of the $\epsilon$-quadratic colour Lie algebra $(\so_{\epsilon}(V,(\phantom{v},\phantom{v})),B)$ where $B(f,g):=-\frac{1}{2}Tr_{\epsilon}(fg)$ for all $f,g \in \so_{\epsilon}(V,(\phantom{v},\phantom{v}))$. Then, the corresponding moment map $\mu_{can}$ satisfies
\begin{equation} \label{canonical moment map}
\mu_{can}(u,v)(w)=\epsilon(v,w)(u,w)v-(v,w)u \qquad \forall u,v,w\in V.
\end{equation}
\end{pp}

\begin{proof}
By Lemma \ref{lm canonical moment map} we have
$$\mu_{can}(u,v)=f(u,v) \qquad \forall u,v\in V$$
and so
$$\mu_{can}(u,v)(w)=\epsilon(v,w)(u,w)v-(v,w)u \qquad \forall u,v,w \in V.$$
\end{proof}

We now calculate $\mu_{can}$ for the standard symplectic plane.
\vspace{0.1cm}

\begin{ex} \label{sl2 moment map}
Let $\lbrace p,q \rbrace$ be the canonical basis of $k^2$ and let $\omega$ be the symplectic form on $k^2$ defined by $\omega(p,q)=1$. The Lie algebra $\sp(k^2,\omega)$ is just $\sl(2,k)$ and $(k^2,\omega)$ is an $\epsilon$-orthogonal representation, where $\Gamma=\mathbb{Z}_2$, $k^2$ is $\mathbb{Z}_2$-graded by $(k^2)_0=\lbrace 0 \rbrace$, $(k^2)_1=k^2$ and $\epsilon(a,b)=(-1)^{ab}$ for all $a,b \in \mathbb{Z}_2$. The Lie algebra $\sl(2,k)$ is quadratic with respect to the form $\frac{1}{2}Tr(XY)$ for all $X,Y \in\sl(2,k)$ and the canonical moment map is given by
$$\mu_{can}(p,p)=\begin{pmatrix}
0 & -2 \\
0 & 0
\end{pmatrix}, \quad \mu_{can}(q,q)=\begin{pmatrix}
0 & 0 \\
2 & 0
\end{pmatrix}, \quad \mu_{can}(p,q)=\begin{pmatrix}
1 & 0 \\
0 & -1
\end{pmatrix}.$$
\end{ex}
\vspace{0.1cm}

\begin{rem} \label{rem mucan iso}
In fact, with the appropriate definition of the $\epsilon$-exterior algebra of a $\Gamma$-graded vector space together with a commutation factor, $\mu_{can}$ factors through an equivariant isomorphism of $\Lambda^2_{\epsilon}(V)$ with $\so_{\epsilon}(V,(\phantom{v},\phantom{v}))$ (see \cite{ChenKang16} if $char(k)=0$ or chapter 3 of \cite{MeyerThesis}). This generalises the fact that if $(V,\omega)$ is a symplectic vector space, then $S^2(V)$ is isomorphic to $\sp(V,\omega)$. It also includes the fact that if $(V,(\phantom{v},\phantom{v}))$ is a quadratic vector space, then $\Lambda^2(V)$ is isomorphic to $\so(V,(\phantom{v},\phantom{v}))$.
\end{rem}
\vspace{0.1cm}

We now give a formula for the moment map of a tensor product of $\epsilon$-orthogonal representations in terms of the moment maps of the factors.
\vspace{0.2cm}

Let $(\gg,B_{\gg})$ and $(\hh,B_{\hh})$ be finite-dimensional $\epsilon$-quadratic colour Lie algebras with respect to $(\Gamma,\epsilon)$ and let
$$\rho_{\gg} : \gg \rightarrow \so_{\epsilon}(V,(\phantom{v},\phantom{v})_V),\qquad \rho_{\hh} : \hh \rightarrow \so_{\epsilon}(W,(\phantom{v},\phantom{v})_W)$$
be finite-dimensional $\epsilon$-orthogonal representations. We denote by $\mu_V \in Alt_{\epsilon}^2(V,\gg)$ and $\mu_W \in Alt_{\epsilon}^2(W,\hh)$ the corresponding moment maps.
\vspace{0.2cm}

On the vector space $V\otimes W$ we define the bilinear form
$$(v\otimes w,v'\otimes w')_{V\otimes W}:=\epsilon(w,v')(v,v')_V (w,w')_W \qquad \forall v\otimes w,v'\otimes w' \in V\otimes W.$$

\begin{pp}
With the notation above, the bilinear form $(\phantom{v},\phantom{v})_{V\otimes W}$ is non-degenerate and $\epsilon$-symmetric for the natural $\Gamma$-grading of $V\otimes W$.
\end{pp}

\begin{proof}
Let $v\otimes w,v'\otimes w' \in V\otimes W$. We have
\begin{align*}
(v\otimes w,v'\otimes w')_{V\otimes W}&=\epsilon(w,v')(v,v')_V (w,w')_W\\
&=\epsilon(w,v')\epsilon(v,v')\epsilon(w,w')(v',v)_V (w',w)_W\\
&=\epsilon(v+w,v'+w')\epsilon(w',v)(v',v)_V (w',w)_W\\
&=\epsilon(v\otimes w,v'\otimes w')(v'\otimes w', v\otimes w)_{V\otimes W}
\end{align*} 
and so $(\phantom{v},\phantom{v})_{V\otimes W}$ is $\epsilon$-symmetric. It is non-degenerate as a direct consequence of the fact that the tensor product of non-degenerate bilinear forms is non-degenerate.
\end{proof}

Hence, one can define an $\epsilon$-orthogonal representation of the $\epsilon$-quadratic colour Lie algebra $(\gg\oplus \hh, B_{\gg}\perp B_{\hh})$
$$\rho : \gg\oplus \hh \rightarrow \so_{\epsilon}(V\otimes W,(\phantom{v},\phantom{v})_{V\otimes W})$$
by:
\begin{equation} \label{tens prod of rep}
\rho(g+h)(v\otimes w):= \rho_{\gg}(g)(v)\otimes w + \epsilon(h,v)v\otimes \rho_{\hh}(w) \quad \forall g\in \gg, ~\forall h \in \hh,~\forall v\otimes w \in V\otimes W.
\end{equation}

As we now show, the moment map of the tensor product is essentially the product of the moment maps of the factors.

\begin{pp}
With the notation above, the moment map $\mu_{V\otimes W} \in Alt_{\epsilon}^2(V\otimes W,\gg\oplus \hh)$ satisfies:
\begin{equation}\label{moment map of tens prod}
\mu_{V\otimes W}(v\otimes w, v'\otimes w')=\epsilon(w,v')\Big( \mu_V(v,v')(w,w')_W+(v,v')_V\mu_W(w,w')\Big)
\end{equation}
for all $v\otimes w$, $v'\otimes w'$ in $V\otimes W$.
\end{pp}

\begin{proof}
Let $g \in \gg$, $h\in \hh$ and $v\otimes w,v'\otimes w'\in V\otimes W$. We have
\begin{align*}
B_{\gg\oplus\hh}(g+h,\mu_{V\otimes W}(v\otimes w, v'\otimes w'))=&((g+h)(v\otimes w),v'\otimes w')_{V\otimes W} \\
=&(g(v)\otimes w,v'\otimes w')_{V\otimes W}+\epsilon(h,v)(v\otimes h(w),v'\otimes w')_{V\otimes W}\\
=&\epsilon(w,v')(g(v),v')_V(w,w')_W\\
&+\epsilon(h+w,v')\epsilon(h,v)(v,v')_V(h(w),w')_W\\
=&\epsilon(w,v')(B_{\gg}(g,\mu_V(v,v')(w,w')_W)\\
&+B_{\hh}(h,\epsilon(h,v+v')\mu_W(w,w')(v,v')_V))
\end{align*}
and since $(v,v')_V=0$ if $|v|\neq -|v'|$, we obtain
$$B_{\gg\oplus\hh}(g+h,\mu_{V\otimes W}(v\otimes w, v'\otimes w'))=B_{\gg\oplus\hh}\big(g+h,\epsilon(w,v')( \mu_V(v,v')(w,w')_W+\mu_W(w,w')(v,v')_V )\big).$$
\end{proof}

\section{Characterisation of $\epsilon$-orthogonal representations of colour Lie type and the norm of the moment map} \label{section colour Lie type}

Let $\Gamma$ be an abelian group and let $\epsilon$ be a commutation factor of $\Gamma$. In this section, we state and prove a theorem on $\epsilon$-orthogonal representations of $\epsilon$-quadratic colour Lie algebras which contains and generalises the results of Kostant (see \cite{Kos99}, \cite{Kos01}) on orthogonal and symplectic complex representations of quadratic Lie algebras and the results of Chen and Kang (see \cite{ChenKang2015}) on orthosymplectic complex representations of quadratic Lie superalgebras.

\begin{df}
Let $\rho : \gg \rightarrow \so_{\epsilon}(V,(\phantom{v},\phantom{v}))$ be a finite-dimensional $\epsilon$-orthogonal representation of a finite-dimensional $\epsilon$-quadratic colour Lie algebra $(\gg,B_{\gg})$. Let $\gt:=\gg\oplus V$ and let $B_{\gt}$ be the non-degenerate $\epsilon$-symmetric bilinear form on $\gt$ defined by $B_{\gt}:=B_{\gg}\perp (\phantom{v},\phantom{v})$. We say that the representation $\rho : \gg \rightarrow \so_{\epsilon}(V,(\phantom{v},\phantom{v}))$ is of colour Lie type if there exists a colour Lie algebra structure $\lbrace \phantom{v},\phantom{v} \rbrace$ on $\gt$ such that
\begin{itemize}
\item the $\epsilon$-quadratic form $B_{\gt}$ is $ad(\gt)$-invariant ;
\item $\lbrace x,y\rbrace=\lbrace x,y\rbrace_{\gg}$ for $x,y$ in $\gg$ ;
\item $\lbrace x,v \rbrace=\rho(x)(v)$ for $x$ in $\gg$, for $v$ in $V$.
\end{itemize}
If we also have $\lbrace V,V \rbrace\subseteq \gg$ then the representation $\rho : \gg \rightarrow \so_{\epsilon}(V,(\phantom{v},\phantom{v}))$ is said to be of colour $\mathbb{Z}_2$-Lie type.
\end{df}

\begin{ex} \label{ex nat rep of so is Lie type}
Let $V$ be a finite-dimensional $\Gamma$-graded vector space together with a non-degenerate $\epsilon$-symmetric bilinear form $(\phantom{v},\phantom{v})$. Then, $V$ is of colour $\mathbb{Z}_2$-Lie type as a representation of $(\so_{\epsilon}(V,(\phantom{v},\phantom{v})),B)$ where
\begin{align*}
B(f,g)&:=-\frac{1}{2}Tr_{\epsilon}(fg) \qquad \forall f,g \in \so_{\epsilon}(V,(\phantom{v},\phantom{v})), \\
\lbrace v,w \rbrace&:=\mu_{can}(v,w) \qquad \forall v,w \in V.
\end{align*}
The colour Lie algebra $\so_{\epsilon}(V,(\phantom{v},\phantom{v}))\oplus V$ is then isomorphic to the colour Lie algebra $\so_{\epsilon}(V\oplus L,(\phantom{v},\phantom{v})\perp (\phantom{v},\phantom{v})_{L} )$ where $L$ is a one-dimensional trivially $\Gamma$-graded vector space with an appropriate non-degenerate $\epsilon$-symmetric bilinear form $(\phantom{v},\phantom{v})_{L}$.
\end{ex}
\vspace{0.1cm}

Recall that (see Definition \ref{df exterior product}) the bilinear form $B_{\gt}$ on $\gt=\gg\oplus V$ allows us to define an exterior product
$$\wedge_{B_{\gt}} : Alt_{\epsilon}(V,\gt)\times Alt_{\epsilon}(V,\gt) \rightarrow Alt_{\epsilon}(V)$$
and hence the norm of any $f\in Alt_{\epsilon}(V,\gt)$ by $N_{B_{\gt}}(f)=f\wedge_{B_{\gt}} f$ (see Definition \ref{def norm}).
\vspace{0.1cm}

\begin{thm} \label{thm N(mu+phi)=0}
Let $\rho : \gg \rightarrow \so_{\epsilon}(V,(\phantom{v} , \phantom{v}))$ be a finite-dimensional $\epsilon$-orthogonal representation of a finite-dimensional $\epsilon$-quadratic colour Lie algebra $(\gg,B_{\gg})$ and let $\mu \in Alt_{\epsilon}^2(V,\gg)$ be its moment map (see Definition \ref{df moment map}).
\begin{enumerate}[label=\alph*)]
\item If $\rho : \gg \rightarrow \so_{\epsilon}(V,(\phantom{v} , \phantom{v}))$ is of colour Lie type, then $\phi \in Alt^2_{\epsilon}(V,\gg\oplus V)$ defined by
$$\phi(v,w)=\lbrace v,w \rbrace - \mu(v,w) \qquad \forall v,w\in V$$
is of degree $0$, takes its values in $V$ and satisfies:
\begin{align}
\rho(x)(\phi(v,w))&=\phi(\rho(x)(v),w)+\epsilon(x,v)\phi(v,\rho(x)(w)) \quad &\forall x \in \gg, ~ \forall v,w \in V \label{phi g-inv}, \\
(\phi(u,v),w)&=-\epsilon(u,v)(v,\phi(u,w)) \quad &\forall u,v,w \in V \label{phi ()-inv}.
\end{align}
\item Let $\phi \in Alt_{\epsilon}^2(V,V)$ be of degree $0$ and satisfy \eqref{phi g-inv} and \eqref{phi ()-inv}. Let $\gt:=\gg\oplus V$, let $B_{\gt}:=B_{\gg}\perp (\phantom{v},\phantom{v})$ and let $\lbrace \phantom{v},\phantom{v} \rbrace \in Alt^2_{\epsilon}(\gt,\gt)$ be the unique map which extends the bracket of $\gg$, the action of $\gg$ on $V$ and such that
$$\lbrace v,w\rbrace=\mu(v,w)+\phi(v,w) \qquad \forall v,w \in V.$$
Then the following are equivalent:
\begin{enumerate}[label=\roman*)]
\item $(\gt,B_{\gt},\lbrace \phantom{v},\phantom{v} \rbrace)$ is an $\epsilon$-quadratic colour Lie algebra.
\item $N_{B_{\gt}}(\mu +\phi)=0$.
\item $N_{B_{\gt}}(\mu)=-N_{B_{\gt}}(\phi)$.
\end{enumerate}
\end{enumerate}
\end{thm}
\vspace{0.1cm}

\begin{proof}
$a)$ Clearly, $\lbrace v,w \rbrace - \mu(v,w) \in V$ since $B_{\gg}$ is orthogonal to $(\phantom{v},\phantom{v})$ and
$$B_{\gt}(x,\lbrace v,w \rbrace)=B_{\gt}(x,\mu(v,w))\qquad \forall x \in \gg.$$
Equation \eqref{phi g-inv} follows from the $\epsilon$-Jacobi identity of $\gt$ and \eqref{phi ()-inv} from the $ad$-invariance of $B_{\gt}$.
\vspace{0.5cm}

$b)$ A straightforward calculation shows that $B_{\gt}$ is ad($\gt$)-invariant in the sense of Definition \ref{def quadratic}. Moreover, for $u,v,w\in\gt$, if $u,v$ or $w$ is an element of $\gg$ then
$$\epsilon(w,u)\lbrace \lbrace u,v \rbrace , w \rbrace+\epsilon(u,v)\lbrace \lbrace v,w \rbrace , u \rbrace +\epsilon(v,w)\lbrace \lbrace w,u \rbrace , v \rbrace=0$$
since $V$ is $\gg$-equivariant. We need the following lemma:
\begin{lm}
Let $u,v,w \in V$. We have
$$B_{\gg}\Big(\epsilon(w,u)\lbrace \lbrace u,v \rbrace , w \rbrace+\epsilon(u,v)\lbrace \lbrace v,w \rbrace , u \rbrace +\epsilon(v,w)\lbrace \lbrace w,u \rbrace , v \rbrace,x\Big)=0 \qquad \forall x \in \gg.$$
\end{lm}

\begin{proof}
Let $x \in \gg$. We have
\begin{align*}
B_{\gg}(\lbrace u , \lbrace v,w \rbrace \rbrace,x)&=\epsilon(u+v+w,x)B_{\gg}(x,\lbrace u , \lbrace v,w \rbrace \rbrace) \\
&=-\epsilon(u,v+w)\epsilon(u+v+w,x)B_{\gg}(x,\lbrace \lbrace v,w \rbrace , u \rbrace)\\
&=-\epsilon(u,v+w)\epsilon(u+v+w,x)(\lbrace x, \lbrace v,w \rbrace \rbrace, u)\\
&=-\epsilon(u,v+w)\epsilon(u+v+w,x)(x( \phi (v,w)), u)
\end{align*}
and since 
$$\rho(x)(\phi(v,w))=\phi(\rho(x)(v),w)+\epsilon(x,v)\phi(v,\rho(x)(w)),$$
then
\begin{align*}
B_{\gg}(\lbrace u , \lbrace v,w \rbrace \rbrace,x)&=-\epsilon(u,v+w)\epsilon(u+v+w,x)\Big(( \phi(x(v),w)  , u)+\epsilon(x,v)(\phi(v,x(w)),u)\Big)\\
&=-\epsilon(u,v+w)\epsilon(u+v+w,x)\Big(( \lbrace x(v),w \rbrace  , u)+\epsilon(x,v)(\lbrace v,x(w) \rbrace ,u)\Big)\\
&=-\epsilon(u,v+w)\epsilon(u+v+w,x)\Big(-\epsilon(x+v,w)( \lbrace w, x(v) \rbrace  , u)+\epsilon(x,v)(\lbrace v,x(w) \rbrace ,u)\Big)\\
&=-\epsilon(u,v+w)\epsilon(u+v+w,x)\Big(( x(v) , \lbrace w,u\rbrace )-\epsilon(x,v)\epsilon(v,x+w)(x(w),\lbrace v,u\rbrace)\Big)\\
&=-\epsilon(u,v+w)\epsilon(u+v+w,x)\Big(B_{\gg}( x , \lbrace v,\lbrace w,u\rbrace \rbrace)-\epsilon(v,w)B_{\gg}(x,\lbrace w,\lbrace v,u\rbrace \rbrace)\Big)\\
&=-\epsilon(u,v+w)B_{\gg}( \lbrace v,\lbrace w,u\rbrace \rbrace,x)+\epsilon(u,v+w)\epsilon(v,w)B_{\gg}(\lbrace w,\lbrace v,u\rbrace \rbrace,x).
\end{align*}
Hence,
$$B_{\gg}\Big(\epsilon(w,u)\lbrace \lbrace u,v \rbrace , w \rbrace+\epsilon(u,v)\lbrace \lbrace v,w \rbrace , u \rbrace +\epsilon(v,w)\lbrace \lbrace w,u \rbrace , v \rbrace,x\Big)=0.$$
\end{proof}

From the previous lemma, it follows that $\gt$ is a colour Lie algebra if and only if
$$\Big(\epsilon(v_3,v_1)\lbrace \lbrace v_1,v_2 \rbrace , v_3 \rbrace+\epsilon(v_1,v_2)\lbrace \lbrace v_2,v_3 \rbrace , v_1 \rbrace +\epsilon(v_2,v_3)\lbrace \lbrace v_3,v_1 \rbrace , v_2 \rbrace,v_4\Big)=0 \quad \forall v_1,v_2,v_3,v_4 \in V.$$
We set $\psi:=\mu+\phi \in Alt^2_{\epsilon}(V,\gt)$ and consider $N_{B_{\gt}}(\psi)=\psi\wedge_{B_{\gt}}\psi \in Alt^4_{\epsilon}(V)$. Let $v_1,v_2,v_3,v_4 \in V$. We have
\begin{align*}
N_{B_{\gt}}(\psi)(v_1,v_2,v_3,v_4)=\sum \limits_{\sigma \in S(\lbrace 1,2\rbrace,\lbrace 3,4 \rbrace)} p(\sigma;v_1,v_2,v_3,v_4)B_{\gt}(\psi(v_{\sigma(1)},v_{\sigma(2)}),\psi(v_{\sigma(3)},v_{\sigma(4)})).
\end{align*}
Since
$$S(\lbrace 1,2\rbrace,\lbrace 3,4 \rbrace)=\lbrace id,(123),(1243),(23),(13)(24),(243) \rbrace$$
and
$$B_{\gt}(x,y)=\epsilon(x,y)B_{\gt}(y,x) \qquad \forall x,y\in \gt,$$
we obtain
\begin{align*}
N_{B_{\gt}}(\psi)(v_1,v_2,v_3,v_4)=&B_{\gt}(\psi(v_1,v_2),\psi(v_3,v_4))+\epsilon(v_1,v_2+v_3)B_{\gt}(\psi(v_2,v_3),\psi(v_1,v_4))\\
&-\epsilon(v_3,v_4)\epsilon(v_1,v_2+v_4)B_{\gt}(\psi(v_2,v_4),\psi(v_1,v_3))\\
&-\epsilon(v_2,v_3)B_{\gt}(\psi(v_1,v_3),\psi(v_2,v_4))\\
&+\epsilon(v_1+v_2,v_3+v_4)B_{\gt}(\psi(v_3,v_4),\psi(v_1,v_2))\\
&+\epsilon(v_2+v_3,v_4)B_{\gt}(\psi(v_1,v_4),\psi(v_2,v_3))\\
=&2\Big(B_{\gt}(\psi(v_1,v_2),\psi(v_3,v_4))+\epsilon(v_1,v_2+v_3)B_{\gt}(\psi(v_2,v_3),\psi(v_1,v_4))\\
&-\epsilon(v_2,v_3)B_{\gt}(\psi(v_1,v_3),\psi(v_2,v_4)) \Big).
\end{align*}
Moreover, since
\begin{align*}
B_{\gt}(\psi(v_1,v_2),\psi(v_3,v_4))&=B_{\gg}(\mu(v_1,v_2),\mu(v_3,v_4))+(\phi(v_1,v_2),\phi(v_3,v_4))\\
&=(\mu(v_1,v_2)(v_3),v_4)+\epsilon(v_1+v_2,v_3+v_4)(\phi(v_3,v_4),\phi(v_1,v_2))\\
&=(\mu(v_1,v_2)(v_3),v_4)+(\phi(\phi(v_1,v_2),v_3),v_4)\\
&=(\lbrace \lbrace v_1,v_2\rbrace,v_3 \rbrace,v_4),
\end{align*}
we have
$$N_{B_{\gt}}(\psi)(v_1,v_2,v_3,v_4)=2\epsilon(v_1,v_3)\Big(\epsilon(v_3,v_1)\lbrace \lbrace v_1,v_2 \rbrace,v_3\rbrace+\epsilon(v_1,v_2)\lbrace \lbrace v_2,v_3\rbrace , v_1\rbrace+\epsilon(v_2,v_3)\lbrace \lbrace v_3,v_1\rbrace , v_2\rbrace,v_4\Big).$$
As pointed out above $\gt=\gg\oplus V$ is a colour Lie algebra if and only if
$$(\epsilon(v_3,v_1)\lbrace \lbrace v_1,v_2 \rbrace , v_3 \rbrace+\epsilon(v_1,v_2)\lbrace \lbrace v_2,v_3 \rbrace , v_1 \rbrace +\epsilon(v_2,v_3)\lbrace \lbrace v_3,v_1 \rbrace , v_2 \rbrace,v_4)=0 \quad \forall v_1,v_2,v_3,v_4 \in V$$
and so this is equivalent to $N_{B_{\gt}}(\mu + \phi)=0$. This proves $i)$ is equivalent to $ii)$. Note that since $V$ is orthogonal to $\gg$, we have
$$\mu \wedge_{B_{\gt}} \phi=0,$$
hence
$$N_{B_{\gt}}(\mu + \phi)=N_{B_{\gt}}(\mu) + N_{B_{\gt}}(\phi)$$
and so $ii)$ is equivalent to $iii)$.
\end{proof}
\vspace{0.1cm}

Taking $\phi=0$, the theorem implies the following characterisations of representations of colour $\mathbb{Z}_2$-Lie type.
\vspace{0.1cm}

\begin{cor}
Let  $\rho : \gg \rightarrow \so_{\epsilon}(V,(\phantom{v} , \phantom{v}))$ be a finite-dimensional $\epsilon$-orthogonal representation of a finite-dimensional $\epsilon$-quadratic colour Lie algebra $(\gg,B_{\gg})$ and let $\mu \in Alt_{\epsilon}^2(V,\gg)$ be its moment map (see Definition \ref{df moment map}). Let $\gt:=\gg\oplus V$, let $B_{\gt}:=B_{\gg}\perp (\phantom{v},\phantom{v})$ and let $\lbrace \phantom{v},\phantom{v} \rbrace \in Alt^2_{\epsilon}(\gt,\gt)$ be the unique map which extends the bracket of $\gg$, the action of $\gg$ on $V$ and such that
$$\lbrace v,w\rbrace=\mu(v,w) \qquad \forall v,w \in V.$$
Then the following are equivalent:
\begin{enumerate}[label=\alph*)]
\item $(\gt,B_{\gt},\lbrace \phantom{v},\phantom{v} \rbrace)$ is an $\epsilon$-quadratic colour Lie algebra.
\item $N_{B_{\gt}}(\mu)=0$.
\item $V$ is of colour $\mathbb{Z}_2$-Lie type.
\end{enumerate}
\end{cor}

\section{The Bianchi map and special $\epsilon$-orthogonal representations} \label{section bianchi}

Let $\Gamma$ be an abelian group and let $\epsilon$ be a commutation factor of $\Gamma$. In the previous section we saw (Theorem \ref{thm N(mu+phi)=0}) that one can associate a colour Lie algebra $\gt$ to certain data on an $\epsilon$-orthogonal representation $V$ of an $\epsilon$-quadratic colour Lie algebra if a particular invariant of this data vanishes. In this section we will give a different interpretation of the data as an element of a space of ``curvature tensors'' on $V$. We show that the vanishing of the above invariant is equivalent to an algebraic ``Bianchi identity'' for the corresponding curvature tensor. Further analysis of this identity makes it clear that there are other natural conditions we can impose on $\epsilon$-orthogonal representations and this leads to the notion of special $\epsilon$-orthogonal representations.
\vspace{0.1cm}

\begin{df}
Let $V$ be a $\Gamma$-graded vector space. We define $\mathcal{R}(V)$ to be the vector space of all multilinear maps $R : V\times V \times V \times V \rightarrow k$ which satisfy
\begin{align}
R(A,B,C,D)&=-\epsilon(A,B)R(B,A,C,D) \qquad \forall A,B,C,D \in V, \label{R prop 1} \\
R(A,B,C,D)&=\epsilon(A+B,C+D)R(C,D,A,B) \qquad \forall A,B,C,D \in V. \label{R prop 2}
\end{align}
\end{df}

\begin{rem}
A map $R\in \mathcal{R}(V)$ satisfies
$$R(A,B,C,D)=-\epsilon(C,D)R(A,B,D,C) \qquad \forall A,B,C,D \in V.$$
\end{rem}
\vspace{0.1cm}

In general a map $R \in \mathcal{R}(V)$ is not $\epsilon$-alternating but we can define the Bianchi map $\beta : \mathcal{R}(V) \rightarrow \mathcal{R}(V)$ which has the property that $\beta(R) \in Alt_{\epsilon}^4(V)$ for all $R \in \mathcal{R}(V)$.
\vspace{0.1cm}

\begin{df} Let $V$ be a $\Gamma$-graded vector space.
\begin{enumerate}[label=\alph*)]
\item The Bianchi map $\beta : \mathcal{R}(V) \rightarrow \mathcal{R}(V)$ is defined by
$$\beta(R)(A,B,C,D):=R(A,B,C,D)+\epsilon(A,B+C)R(B,C,A,D)+\epsilon(A+B,C)R(C,A,B,D)$$
for all $A,B,C,D \in V$.
\item The vector space $Ker(\beta)$ is called the space of formal curvature tensors of $V$.
\end{enumerate}
\end{df}
\vspace{0.1cm}

\begin{pp} \label{pp projection curvature} Let $V$ be a $\Gamma$-graded vector space and let $R \in \mathcal{R}(V)$. Then $\frac{1}{3}\beta(R)$ is the projection of $R$ onto $Alt_{\epsilon}^4(V)$ parallel to $Ker(\beta)$.
\end{pp}

\begin{proof}
Let $A,B,C,D \in V$. We show that $\beta(R)$ is an $\epsilon$-alternating multilinear map. By Remark \ref{rem def equivalente alternating map} $\beta(R)$ is $\epsilon$-alternating if and only if
\begin{align*}
\beta(R)(A,B,C,D)&=-\epsilon(A,B)\beta(R)(B,A,C,D), \\
\beta(R)(A,B,C,D)&=-\epsilon(B,C)\beta(R)(A,C,B,D), \\
\beta(R)(A,B,C,D)&=-\epsilon(C,D)\beta(R)(A,B,D,C)
\end{align*}
for all $A,B,C,D \in V$ and these identities can be shown by straightforward calculations. Moreover, if $R\in Alt_{\epsilon}^4(V)$, we have
\begin{align*}
\epsilon(A,B+C)R(B,C,A,D)&=R(A,B,C,D), \\
\epsilon(A+B,C)R(C,A,B,D)&=R(A,B,C,D),
\end{align*}
and so $\beta(R)=3R$. Hence the map $R \mapsto \frac{1}{3}\beta(R)$ for $R \in \mathcal{R}(V)$ is a projection.
\end{proof}
\vspace{0.1cm}

We now show that, to the data of Theorem \ref{thm N(mu+phi)=0} one can associate an element of $\mathcal{R}(V)$ and interpret the vanishing condition therein in terms of the Bianchi map $\beta$.
\vspace{0.1cm}

\begin{df}
Let  $\rho : \gg \rightarrow \so_{\epsilon}(V,(\phantom{v} , \phantom{v}))$ be a finite-dimensional $\epsilon$-orthogonal representation of a finite-dimensional $\epsilon$-quadratic colour Lie algebra $(\gg,B_{\gg})$, let $\mu \in Alt_{\epsilon}^2(V,\gg)$ be its moment map and let $\phi \in Alt_{\epsilon}^2(V,V)$ be of degree $0$ and satisfy \eqref{phi g-inv} and \eqref{phi ()-inv}. Let $\gt:=\gg\oplus V$, let $B_{\gt}:=B_{\gg}\perp (\phantom{v},\phantom{v})$ and let $\lbrace \phantom{v},\phantom{v} \rbrace \in Alt^2_{\epsilon}(\gt,\gt)$ be the unique map which extends the bracket of $\gg$, the action of $\gg$ on $V$ and such that
$$\lbrace v,w\rbrace=\mu(v,w)+\phi(v,w) \qquad \forall v,w \in V.$$
Define the multilinear map $R_{\mu+\phi} : V\times V \times V \times V \rightarrow k$ by
$$R_{\mu+\phi}(A,B,C,D):=B_{\gt}(\lbrace \lbrace A,B \rbrace, C\rbrace ,D) \qquad \forall A,B,C,D \in V.$$
\end{df}
\vspace{0.1cm}

\begin{pp} \label{pp norme egal bianchi}
With the notation above:
\begin{enumerate}[label=\alph*)]
\item The map $R_{\mu+\phi}$ is an element of $\mathcal{R}(V)$.
\item We have $N_{B_{\gt}}(\mu+\phi)=2\beta(R_{\mu+\phi})$.
\item The following are equivalent:
\begin{enumerate}[label=\roman*)]
\item $(\gt,B_{\gt},\lbrace \phantom{v},\phantom{v} \rbrace)$ is an $\epsilon$-quadratic colour Lie algebra.
\item $\beta(R_{\mu+\phi})=0$.
\end{enumerate}
\end{enumerate}
\end{pp}

\begin{proof}
$a)$ Since the bracket is $\epsilon$-antisymmetric, \eqref{R prop 1} is satisfied. For $A,B,C,D \in V$, we have
\begin{align*}
R_{\mu+\phi}(A,B,C,D)&=(\mu(A,B)(C),D)+(\phi(\phi(A,B),C),D)\\
&=B_{\gg}(\mu(A,B),\mu(C,D))+(\phi(A,B),\phi(C,D))\\
&=\epsilon(A+B,C+D)\Big(B_{\gg}(\mu(C,D),\mu(A,B))+(\phi(C,D),\phi(A,B))\Big)\\
&=\epsilon(A+B,C+D)\Big((\mu(C,D)(A),B)+(\phi(\phi(C,D),A),B)\Big)\\
&=\epsilon(A+B,C+D)R_{\mu+\phi}(C,D,A,B)
\end{align*}
and then Equation \eqref{R prop 2} is satisfied.
\vspace{0.2cm}

\noindent
$b)$ For $A,B,C,D\in V$, we have
\begin{align*}
&\beta(R_{\mu+\phi})(A,B,C,D)\\
&=R_{\mu+\phi}(A,B,C,D)+\epsilon(A,B+C)R_{\mu+\phi}(B,C,A,D)+\epsilon(A+B,C)R_{\mu+\phi}(C,A,B,D)\\
&=\Big(\lbrace \lbrace A,B \rbrace, C\rbrace+\epsilon(A,B+C)\lbrace \lbrace B,C \rbrace, A\rbrace+\epsilon(A+B,C)\lbrace \lbrace C,A \rbrace, B\rbrace ,D\Big)\\
&=\epsilon(A,C)\Big(\epsilon(C,A)\lbrace\lbrace A,B\rbrace,C\rbrace+\epsilon(A,B)\lbrace\lbrace B,C\rbrace,A\rbrace+\epsilon(B,C)\lbrace\lbrace C,A\rbrace,B\rbrace ,D\Big).
\end{align*}
As we have seen in the proof of the theorem \ref{thm N(mu+phi)=0} we have
$$N_{B_{\gt}}(\mu+\phi)(A,B,C,D)=2\epsilon(A,C)(\epsilon(C,A)\lbrace\lbrace A,B\rbrace,C\rbrace+\epsilon(A,B)\lbrace\lbrace B,C\rbrace,A\rbrace+\epsilon(B,C)\lbrace\lbrace C,A\rbrace,B\rbrace,D)$$
and so $N_{B_{\gt}}(\mu+\phi)=2\beta(R_{\mu+\phi})$.
\vspace{0.2cm}

\noindent
$c)$ Follows from Theorem \ref{thm N(mu+phi)=0} and $b)$ above.
\end{proof}

By Example \ref{ex nat rep of so is Lie type}, Theorem \ref{thm N(mu+phi)=0} and this proposition, $\beta(R_{\mu_{can}})=0$. Since $Id-\frac{1}{3}\beta$ is the projection onto $Ker(\beta)$ by Proposition \ref{pp projection curvature}, it is natural to ask which $\epsilon$-orthogonal representations have curvature $R_{\mu}$ such that $R_{\mu}-\frac{1}{3}\beta(R_{\mu})$ is equal to $R_{\mu_{can}}$.
\vspace{0.1cm}

\begin{df}
Let $\rho : \gg \rightarrow \so_{\epsilon}(V,(\phantom{v},\phantom{v}))$ be a finite-dimensional $\epsilon$-orthogonal representation  of a finite-dimensional $\epsilon$-quadratic colour Lie algebra. The representation $V$ is special if its moment map $\mu$ satisfies
$$R_{\mu}-\frac{1}{3}\beta(R_{\mu})=R_{\mu_{can}}.$$
\end{df}
\vspace{0.1cm}

We now characterise special $\epsilon$-orthogonal representations in terms of their moment map.

\begin{pp} \label{equivalent conditions CS}
Let $\rho : \gg \rightarrow \so_{\epsilon}(V,(\phantom{v},\phantom{v}))$ be a finite-dimensional $\epsilon$-orthogonal representation of a finite-dimensional $\epsilon$-quadratic colour Lie algebra. The following are equivalent:
\begin{enumerate}[label=\alph*)]
\item $V$ is a special $\epsilon$-orthogonal representation.
\item $\mu(A,B)(C)+\epsilon(B,C)\mu(A,C)(B)=(A,B)C+\epsilon(B,C)(A,C)B-2(B,C)A$ 

$\forall A,B,C \in V$. \label{CS}
\item \label{CS 2} $\mu(A,B)(C)+\epsilon(B,C)\mu(A,C)(B)=\mu_{can}(A,B)(C)+\epsilon(B,C)\mu_{can}(A,C)(B)$

$\forall A,B,C \in V$.
\end{enumerate}
\end{pp}

\begin{proof}
Property $b)$ is equivalent to $c)$ by Equation \eqref{canonical moment map}. Let $A,B,C,D\in V$. We now show that $b)$ implies $a)$. We have:
\begin{align*}
&R_{\mu}(A,B,C,D)-\frac{1}{3}\beta(R_{\mu})(A,B,C,D)\\
&=\frac{1}{3}\Big(2\mu(A,B)(C)-\epsilon(A,B+C)\mu(B,C)(A)+\epsilon(B,C)\mu(A,C)(B) ,D\Big)\\
&=\frac{1}{3}\Big(\mu(A,B)(C)+\epsilon(B,C)\mu(A,C)(B)-\epsilon(A,B)(\mu(B,A)(C)+\epsilon(A,C)\mu(B,C)(A))  ,D\Big).
\end{align*}
Now, using $b)$ twice, we have
\begin{align*}
\mu(A,B)(C)+\epsilon(B,C)\mu(A,C)(B)&=(A,B)C+\epsilon(B,C)(A,C)B-2(B,C)A,\\
\mu(B,A)(C)+\epsilon(A,C)\mu(B,C)(A)&=(B,A)C+\epsilon(A,C)(B,C)A-2(A,C)B
\end{align*}
and hence:
\begin{align*}
&R_{\mu}(A,B,C,D)-\frac{1}{3}\beta(R_{\mu})(A,B,C,D)\\
&=\frac{1}{3}\Big((A,B)C+\epsilon(B,C)(A,C)B-2(B,C)A-\epsilon(A,B)((B,A)C+\epsilon(A,C)(B,C)A-2(A,C)B) ,D\Big)\\
&=\Big(\mu_{can}(A,B)(C),D\Big)\\
&=R_{\mu_{can}}(A,B,C,D).
\end{align*}
This proves $a)$. To show that $a)$ implies $b)$, suppose that $R_{\mu}-\frac{1}{3}\beta(R_{\mu})=R_{\mu_{can}}$. Then
\begin{align*}
&\Big((A,B)C+\epsilon(B,C)(A,C)B-2(B,C)A,D \Big)\\
&=\Big(\mu_{can}(A,B)(C)+\epsilon(B,C)\mu_{can}(A,C)(B),D \Big)\\
&=R_{\mu_{can}}(A,B,C,D)+\epsilon(B,C)R_{\mu_{can}}(A,C,B,D)\\
&=\Big(R_{\mu}-\frac{1}{3}\beta(R_{\mu})\Big)(A,B,C,D)+\epsilon(B,C)\Big(R_{\mu}-\frac{1}{3}\beta(R_{\mu})\Big)(A,C,B,D)\\
&=R_{\mu}(A,B,C,D)+\epsilon(B,C)R_{\mu}(A,C,B,D)\\
&=\Big(\mu(A,B)(C)+\epsilon(B,C)\mu(A,C)(B),D \Big),
\end{align*}
and since $(\phantom{v},\phantom{v})$ is non-degenerate, we obtain $b)$.
\end{proof}

\begin{rem}
The equivalent conditions $b)$ and $c)$ define the notion of special $\epsilon$-orthogonal representations over fields of characteristic three and hence certain representations considered in \cite{Elduque06} are special $\epsilon$-orthogonal representations.
\end{rem}

\begin{ex}
Let $\gg$ be a Lie algebra and let $(V,(\phantom{v},\phantom{v}))$ be a symplectic representation of $\gg$. It follows from Proposition \ref{equivalent conditions CS} that $\gg\rightarrow \so_{\epsilon}(V,(\phantom{v},\phantom{v}))=\sp(V,(\phantom{v},\phantom{v}))$ is a special $\epsilon$-orthogonal representation if and only if it is a special symplectic representation in the sense of \cite{StantonSlupinski15}.
\end{ex}
\vspace{0.2cm}

We now investigate when the tensor product of two $\epsilon$-orthogonal representations is of colour $\mathbb{Z}_2$-Lie type if one of them, $W$, is the fundamental representation of an $\epsilon$-orthogonal colour Lie algebra. It turns out that a necessary and sufficient condition for this to be the case is that the other representation, $V$, is also the fundamental representation of an $\epsilon$-orthogonal colour Lie algebra unless $dim(W)=1$ or $dim(W)=2$. If $dim(W)=1$, the necessary and sufficient condition is that $V$ is of colour $\mathbb{Z}_2$-Lie type and if $dim(W)=2$, that $V$ is a special $\epsilon$-orthogonal representation.

\begin{thm} \label{rep tensor prod with sl2}
Suppose that the representation $k^2$ of $\sl(2,k)$ is an $\epsilon$-orthogonal representation with respect to $(\Gamma,\epsilon)$. Let $\rho : \gg \rightarrow \so_{\epsilon}(V,(\phantom{v},\phantom{v})_V)$ be a finite-dimensional faithful $\epsilon$-orthogonal representation of a finite-dimensional $\epsilon$-quadratic colour Lie algebra and let $W$ be a finite-dimensional $\Gamma$-graded vector space together with a non-degenerate $\epsilon$-symmetric bilinear form $(\phantom{v},\phantom{v})_W$. Then the $\epsilon$-orthogonal representation (see Section \ref{section moment map})
$$\gg\oplus \so_{\epsilon}(W,(\phantom{v},\phantom{v})_W) \rightarrow \so_{\epsilon}(V\otimes W,(\phantom{v},\phantom{v})_{V\otimes W})$$
is of colour $\mathbb{Z}_2$-Lie type if and only if one of the following holds:
\begin{enumerate}[label=\alph*)]
\item $\gg$ is isomorphic to $\so_{\epsilon}(V,(\phantom{v},\phantom{v})_V)$ ;
\item $dim(W)=1$, $(\phantom{v},\phantom{v})_W$ is symmetric and $\gg \rightarrow \so_{\epsilon}(V,(\phantom{v},\phantom{v})_V)$ is of colour $\mathbb{Z}_2$-Lie type ;
\item $dim(W)=2$, $(\phantom{v},\phantom{v})_W$ is antisymmetric and $\gg \rightarrow \so_{\epsilon}(V,(\phantom{v},\phantom{v})_V)$ is special $\epsilon$-orthogonal.
\end{enumerate}
\end{thm}

\begin{proof} By \eqref{moment map of tens prod}, the moment map $\mu_{V\otimes W} \in Alt^2_{\epsilon}(V\otimes W,\gg\oplus\hh)$ satisfies for $v\otimes w, v'\otimes w'\in V\otimes W$
$$\mu_{V\otimes W}(v\otimes w, v'\otimes w')=\epsilon(w,v')\Big( \mu_V(v,v')(w,w')_W+\mu_W(w,w')(v,v')_V \Big)$$
where $\mu_V \in Alt_{\epsilon}^2(V,\gg)$ and $\mu_W \in Alt_{\epsilon}^2(W,\hh)$ are the corresponding moment maps of $V$ and $W$. Recall that we have a decomposition $W=W_0\oplus W_1$ (see Proposition \ref{pp V=V0+V1}).
\vspace{0.2cm}

\noindent
$a)$ Suppose that $dim(W_0)+\frac{dim(W_1)}{2}\geq 2$. Then there exists $w,w',w'' \in W$ such that
$$(w,w')_W\neq 0, \qquad (w,w'')_W=0, \qquad (w',w'')_W=0.$$
Let $v,v',v'' \in V$. We now compute the $\epsilon$-Jacobi identity for the three elements $v\otimes w,v'\otimes w', v''\otimes w''$ using Properties \eqref{tens prod of rep}, \eqref{moment map of tens prod} and the moment map of $W$ given by formula \eqref{canonical moment map}.
\vspace{0.2cm}

The first term of this identity is
\begin{align}
&\epsilon(v''+w'',v+w)\mu_{V\otimes W}(v\otimes w,v'\otimes w')(v''\otimes w'') \nonumber \\
&=\epsilon(v''+w'',v+w)\epsilon(w,v')\mu_V(v,v')(v'')\otimes (w,w')_Ww'' \nonumber \\
&=\epsilon(v'',v)\epsilon(v'',w)\epsilon(w'',v)\epsilon(w'',w)\epsilon(w,v')\mu_V(v,v')(v'')\otimes (w,w')_Ww'', \label{jac1}
\end{align}
the second term is
\begin{align}
&\epsilon(v+w,v'+w')\mu_{V\otimes W}(v'\otimes w',v''\otimes w'')(v\otimes w) \nonumber \\
&=\epsilon(v+w,v'-w)\epsilon(-w,v'')\epsilon(-w+w'',v)(v',v'')_V v' \otimes \epsilon(w'',w)\epsilon(w,w)(w,w')_W w'' \nonumber \\
&=\epsilon(v,v')\epsilon(w,v')\epsilon(v'',w)\epsilon(w'',v)\epsilon(w'',w)(v',v'')_V v' \otimes (w,w')_W w'' \label{jac2}
\end{align}
and the last term is
\begin{align}
&\epsilon(v'+w',v''+w'')\mu_{V\otimes W}(v''\otimes w'',v\otimes w)(v'\otimes w') \nonumber \\
&=\epsilon(v'-w,v''+w'')\epsilon(w'',v)\epsilon(w''+w,v')(v'',v)_Vv'\otimes-(w,w')_Ww'' \nonumber \\
&=-\epsilon(v',v'') \epsilon(v'',w)\epsilon(w'',w)\epsilon(w'',v)\epsilon(w,v')(v'',v)_Vv'\otimes(w,w')_Ww''. \label{jac3}
\end{align}
Taking the sum of Equations \eqref{jac1}, \eqref{jac2} and \eqref{jac3} we see that the $\epsilon$-Jacobi identity for the three elements $v\otimes w,v'\otimes w', v''\otimes w''$ is satisfied if and only if
$$\epsilon(v'',v)\mu_V(v,v')(v'')+\epsilon(v,v')(v',v'')_V v'-\epsilon(v',v'')(v'',v)_Vv'=0,$$
and by \eqref{canonical moment map} this is equivalent to
$$\mu_V(v,v')(v'')=\mu_{can}(v,v')(v'').$$
Since $\mu_{can} \in Alt_{\epsilon}^2\Big(V,\so_{\epsilon}(V,(\phantom{v},\phantom{v})_V)\Big)$ is surjective (see Remark \ref{rem mucan iso} and Proposition 3.9.8 of \cite{MeyerThesis}), $\rho ~ \circ ~ \mu \in Alt_{\epsilon}^2(V,\so_{\epsilon}(V,(\phantom{v},\phantom{v})_V))$ is surjective and then $\rho$ is surjective. Since $\rho$ is injective by assumption, $\gg$ is isomorphic to $\so_{\epsilon}(V,(\phantom{v},\phantom{v})_V)$.
\vspace{0.2cm}

\noindent
$b)$ Suppose that $dim(W)=1$. Then we have $\so_{\epsilon}(W,(\phantom{v},\phantom{v})_W)=\lbrace 0 \rbrace$ and the representation
$$\gg\oplus \so_{\epsilon}(W,(\phantom{v},\phantom{v})_W)\rightarrow \so_{\epsilon}(V\otimes W,(\phantom{v},\phantom{v})_{V\otimes W})$$
is of colour $\mathbb{Z}_2$-Lie type if and only if the representation $\gg\rightarrow \so_{\epsilon}(V,(\phantom{v},\phantom{v})_{V})$ is of colour $\mathbb{Z}_2$-Lie type.
\vspace{0.2cm}

\noindent
$c)$ Suppose that $dim(W)=2$ and $(\phantom{v},\phantom{v})_W$ is antisymmetric. Let $\lbrace p,q \rbrace$ be an homogeneous symplectic basis of $(W,(\phantom{v},\phantom{v})_W)$. There exists $\gamma \in \Gamma$ such that $\epsilon(\gamma,\gamma)=-1$, $|p|=\gamma$ and $|q|=-\gamma$ and let $\lbrace E,H,F \rbrace$ be a $\sl(2,k)$-triple of $\so_{\epsilon}(W,(\phantom{v},\phantom{v})_W)$. From Example \ref{sl2 moment map} we have
$$\mu_W(p,p)=-2E,\quad \mu_W(q,q)=2F, \quad \mu_W(p,q)=H.$$

Let $v,v',v'' \in V$. One can check that the $\epsilon$-Jacobi identities of $v\otimes p,v'\otimes p,v''\otimes p$ and $v\otimes q,v'\otimes q,v''\otimes q$ are satisfied. We now compute the $\epsilon$-Jacobi identity for $v\otimes p,v'\otimes p$ and $v''\otimes q$. Its first term is
$$\epsilon(v''+\gamma,v+\gamma)\mu_{V\otimes W}(v\otimes p,v'\otimes p)(v''\otimes q)=2\epsilon(v'',v)\epsilon(\gamma,v'')\epsilon(v,\gamma)^2 (v,v')_V v''\otimes p.$$
The second term is
$$\epsilon(v+\gamma,v'+\gamma)\mu_{V\otimes W}(v'\otimes p,v''\otimes q)(v\otimes p)=\epsilon(v+\gamma,v'+\gamma)\epsilon(\gamma,v'')\Big(\mu_V(v',v'')(v)\otimes p+(v',v'')_V v\otimes p \Big)$$
and the last term is
$$\epsilon(v'+\gamma,v''+\gamma)\mu_{V\otimes W}(v''\otimes q ,v\otimes p)(v'\otimes p)=\epsilon(v'+\gamma,v''-\gamma)\epsilon(v,\gamma)\Big( -\mu_V(v'',v)(v')\otimes p+(v'',v)_Vv'\otimes p \Big).$$
Hence, summing these three terms, $\epsilon$-Jacobi identity for $v\otimes p,v'\otimes p$ and $v''\otimes q$ is satisfied if and only if we have
$$2\epsilon(v'',v)\epsilon(\gamma,v'')\epsilon(v,\gamma)^2 (v,v')_V v'' +\epsilon(v+\gamma,v'+\gamma)(v',v'')_V v+\epsilon(v'+\gamma,v''-\gamma)\epsilon(v,\gamma)(v'',v)_Vv'$$
$$=-\epsilon(v+\gamma,v'+\gamma)\epsilon(\gamma,v'')\mu_V(v',v'')(v)+\epsilon(v'+\gamma,v''-\gamma)\epsilon(v,\gamma) \mu_V(v'',v)(v)$$
and now if we multiply both sides by $\epsilon(v''-\gamma,v'+\gamma)\epsilon(\gamma,v)$, this is equivalent to
$$-2(v,v')_V v''+\epsilon(v,v')(v'',v')_V v+(v'',v)_V v'=-\epsilon(v+v'',v')\mu_V(v',v'')(v)+\mu_V(v'',v)(v')$$
and this is equivalent to
\begin{equation*}
\mu_V(v,v')(v'')+\epsilon(v',v'')\mu_V(v,v'')(v')=(v,v')_V v''+\epsilon(v',v'')(v,v'')_V v'-2(v',v'')_V v. 
\end{equation*}
Finally, by a straightforward calculation, the $\epsilon$-Jacobi identity of $v\otimes p,v'\otimes q,v''\otimes q$ is satisfied if and only if the $\epsilon$-Jacobi identity of $v\otimes p,v'\otimes p,v''\otimes q$ is satisfied.
\end{proof}

\begin{rem}
Colour Lie algebras $\gt$ obtained from special $\epsilon$-orthogonal representations are not arbitrary. They have the particularity to admit a gradation of Heisenberg type (see \cite{Faulkner71} and \cite{StantonSlupinski15}), that is to say that there exists a $H$ in $\gt$ of degree $0$ such that
$$\gt=\gt_{-2}\oplus\gt_{-1}\oplus\gt_0\oplus\gt_1\oplus\gt_2$$
where $\gt_i=\lbrace x\in \gt ~ | ~ \lbrace H,x\rbrace=ix \rbrace$ and $dim(\gt_{-2})=dim(\gt_2)=1$. This property certainly characterises colour Lie algebras that we can obtain from special $\epsilon$-orthogonal representations (for a discussion of this problem in the frame of Lie algebras, see \cite{Faulkner71}, \cite{Cheng87} and \cite{StantonSlupinski15}).
\end{rem}

\section{Mathews identities for the covariants of a special $\epsilon$-orthogonal representation} \label{section covariants}

Let $\Gamma$ be an abelian group and let $\epsilon$ be a commutation factor of $\Gamma$. In this section we introduce and study a bilinear, a trilinear and a quadrilinear $\epsilon$-alternating multilinear map associated to a special $\epsilon$-orthogonal representation. These maps generalise the three classical covariants of the space of binary cubics (see \cite{Eisenstein1844}) which is a special $\epsilon$-orthogonal representation of the Lie algebra $\sl(2,k)$. We will prove a set of identities satisfied by them which generalise the Mathews identities for binary cubics (see \cite{Mathews11}) and their analogues for special symplectic representations of Lie algebras (see \cite{StantonSlupinski15}).
\vspace{0.1cm}

\begin{df} Let $\rho : \gg \rightarrow \so_{\epsilon}(V,(\phantom{v},\phantom{v}))$ be a finite-dimensional $\epsilon$-orthogonal representation of a finite-dimensional $\epsilon$-quadratic colour Lie algebra and let $\mu$ be its moment map. We define the multilinear maps $\psi : V\times V\times V \rightarrow V$ and $Q : V\times V\times V\times V \rightarrow k$ as follows:
\begin{align*}
\psi(v_1,v_2,v_3)&=\mu(v_1,v_2)(v_3)+\epsilon(v_1+v_2,v_3)\mu(v_3,v_1)(v_2)+\epsilon(v_1,v_2+v_3)\mu(v_2,v_3)(v_1), \\
Q(v_1,v_2,v_3,v_4)&=(v_1,\psi(v_2,v_3,v_4))-\epsilon(v_1+v_2+v_3,v_4)(v_4,\psi(v_1,v_2,v_3)) \\
 &+\epsilon(v_1+v_2,v_3+v_4)(v_3,\psi(v_4,v_1,v_2))-\epsilon(v_1,v_2+v_3+v_4)(v_2,\psi(v_3,v_4,v_1))
\end{align*}
for all $v_1,v_2,v_3,v_4 \in V$. The maps $\mu$, $\psi$ and $Q$ are called the covariants of $V$.
\end{df}

In fact, the maps $\mu$, $\psi$ and $Q$ are $\epsilon$-alternating mutilinear maps.

\begin{pp}
Let $\rho : \gg \rightarrow \so_{\epsilon}(V,(\phantom{v},\phantom{v}))$ be a finite-dimensional $\epsilon$-orthogonal representation of a finite-dimensional $\epsilon$-quadratic colour Lie algebra and let $\mu, \psi, Q$ be its covariants. We have $\mu \in Alt^2_{\epsilon}(V,\gg)$, $\psi \in Alt_{\epsilon}^3(V,V)$ and $Q \in Alt_{\epsilon}^4(V)$.
\end{pp}

\begin{proof}
Let $v_1,v_2,v_3,v_4 \in V$. We have
\begin{align}
\psi(v_1,v_2,v_3)&=\frac{1}{2}\sum \limits_{\sigma\in S_3} p(\sigma ; v_1,v_2,v_3) \mu(v_{\sigma(1)},v_{\sigma(2)})(v_{\sigma(3)}), \label{formule Bpsi pour mathews}\\
Q(v_1,v_2,v_3,v_4)&=\frac{1}{2}\sum \limits_{\sigma\in S_4} p(\sigma ; v_1,v_2,v_3,v_4) (v_{\sigma(1)},\mu(v_{\sigma(2)},v_{\sigma(3)})(v_{\sigma(4)})) \label{formule BQ pour mathews}
\end{align}
and so $\psi$ and $Q$ are $\epsilon$-alternating multilinear maps.
\end{proof}

In the case of special $\epsilon$-orthogonal representations, the quadrilinear covariant is essentially the norm of the moment map.

\begin{pp} 
Let $\rho : \gg \rightarrow \so_{\epsilon}(V,(\phantom{v},\phantom{v}))$ be a finite-dimensional special $\epsilon$-orthogonal representation of a finite-dimensional $\epsilon$-quadratic colour Lie algebra $(\gg,B_{\gg})$ and let $\mu, \psi, Q$ be its covariants.
\begin{enumerate}[label=\alph*)]
\item \label{pp formules lien BQ Bpsi dans le cas CS} For all $v_1,v_2,v_3,v_4 \in V$, we have
\begin{align*}
\psi(v_1,v_2,v_3)&=3(\mu(v_1,v_2)(v_3)-\mu_{can}(v_1,v_2)(v_3)), \\
Q(v_1,v_2,v_3,v_4)&=4(v_1,\psi(v_2,v_3,v_4)).
\end{align*}
\item \label{pp BQ=-2N(mu)} We have
$$Q=-2N_{B_{\gg}}(\mu)=-4\beta(R_{\mu}).$$
\end{enumerate}
\end{pp}

\begin{proof}
$a)$ We prove the first identity. Using Proposition \ref{equivalent conditions CS} \ref{CS}, we have
\begin{align*}
\epsilon(v_1+v_2,v_3)\mu(v_3,v_1)(v_2)&=-\epsilon(v_2,v_3)\mu(v_1,v_3)(v_2)\\
&=\mu(v_1,v_2)(v_3)-\epsilon(v_2,v_3)(v_1,v_3)v_2-(v_1,v_2)v_3+2\epsilon(v_2,v_3)(v_3,v_2)v_1
\end{align*}
and
\begin{align*}
\epsilon(v_1,v_2+v_3)\mu(v_2,v_3)(v_1)=&-\epsilon(v_1,v_2)\mu(v_2,v_1)(v_3)+(v_2,v_3)v_1\\
&+\epsilon(v_1,v_2)(v_2,v_1)v_3-2\epsilon(v_1,v_2+v_3)(v_3,v_1)v_2.
\end{align*}
Hence,
\begin{align*}
\psi(v_1,v_2,v_3)=&3\mu(v_1,v_2)(v_3)-\epsilon(v_2,v_3)(v_1,v_3)v_2-(v_1,v_2)v_3+2\epsilon(v_2,v_3)(v_3,v_2)v_1+(v_2,v_3)v_1\\
&+\epsilon(v_1,v_2)(v_2,v_1)v_3-2\epsilon(v_1,v_2+v_3)(v_3,v_1)v_2\\
=&3\mu(v_1,v_2)(v_3)-3\epsilon(v_2,v_3)(v_1,v_3)v_2+3(v_2,v_3)v_1\\
=&3(\mu(v_1,v_2)(v_3)-\mu_{can}(v_1,v_2)(v_3)).
\end{align*}
The second identity can be proved by a similar calculation.
\vspace{0.2cm}

\noindent
$b)$ Let $v_1,v_2,v_3,v_4 \in V$. We know from the proof of Theorem \ref{thm N(mu+phi)=0} that
$$N_{B_{\gg}}(\mu)(v_1, v_2, v_3, v_4)=-2(v_1,\epsilon(v_2,v_3+v_4)\mu(v_3,v_4)(v_2)+\epsilon(v_2+v_3,v_4)\mu(v_4,v_2)(v_3)+\mu(v_2,v_3)(v_4)),$$
by \ref{pp formules lien BQ Bpsi dans le cas CS} we have
$$Q(v_1, v_2, v_3, v_4)=4(v_1,\mu(v_2,v_3)(v_4)+\epsilon(v_2+v_3,v_4)\mu(v_4,v_2)(v_3)+\epsilon(v_2,v_3+v_4)\mu(v_3,v_4)(v_2))$$
and hence $Q=-2N_{B_{\gg}}(\mu)$. Furthermore, by Proposition \ref{pp norme egal bianchi} we also have $Q=-4\beta(R_{\mu})$.
\end{proof}

\begin{ex}
If $\mu=\mu_{can}$ then the covariants $\psi$ and $Q$ are trivial.
\end{ex}
\vspace{0.1cm}

Recall that (see Definition \ref{df exterior product}) the representation $\rho$ allows us to define an exterior product
$$\wedge_{\rho} : Alt_{\epsilon}(V,\gg)\times Alt_{\epsilon}(V,V) \rightarrow Alt_{\epsilon}(V,V)$$
and the scalar multiplication allows us to define exterior products
\begin{alignat*}{3}
&\wedge_{\epsilon} &&: Alt_{\epsilon}(V,k)\times Alt_{\epsilon}(V,k) && \rightarrow Alt_{\epsilon}(V,k),\\
&\wedge_{\times} &&: Alt_{\epsilon}(V,k)\times Alt_{\epsilon}(V,V) && \rightarrow Alt_{\epsilon}(V,V),\\
&\wedge_{\times} &&: Alt_{\epsilon}(V,k)\times Alt_{\epsilon}(V,\gg) && \rightarrow Alt_{\epsilon}(V,\gg).
\end{alignat*}
Recall also that (see Definition \ref{df exterior composition}) we have composition maps
\begin{alignat*}{2}
&\circ : Alt_{\epsilon}(V,\gg) && \times Alt_{\epsilon}(V,V) \rightarrow Alt_{\epsilon}(V,\gg),\\
&\circ : Alt_{\epsilon}(V,V) && \times Alt_{\epsilon}(V,V) \rightarrow Alt_{\epsilon}(V,V),\\
&\circ : Alt_{\epsilon}(V,k) && \times Alt_{\epsilon}(V,V) \rightarrow Alt_{\epsilon}(V,k).
\end{alignat*}

With this notation we now prove some identities satisfied by the covariants of a special $\epsilon$-orthogonal representation which generalise the classical Mathews identities (see \cite{Mathews11}).

\begin{thm} \label{thm id mathews}
Let $\rho : \gg \rightarrow \so_{\epsilon}(V,(\phantom{v},\phantom{v}))$ be a finite-dimensional special $\epsilon$-orthogonal representation of a finite-dimensional $\epsilon$-quadratic colour Lie algebra and let $\mu \in Alt^2_{\epsilon}(V,\gg)$, $\psi \in Alt_{\epsilon}^3(V,V)$ and $Q \in Alt_{\epsilon}^4(V)$ be its covariants. We have the following identities:
\begin{alignat}{4}
a)&\qquad\qquad\qquad\qquad&\mu\wedge_{\rho} \psi&=-\frac{3}{2}Q\wedge_{\times} Id_V \quad &&\text{ in } Alt^5_{\epsilon}(V,V), \label{id premat} \\
b)&\qquad\qquad\qquad\qquad&\mu \circ \psi &= 3Q \wedge_{\times} \mu \quad &&\text{ in } Alt^6_{\epsilon}(V,\gg), \label{id mat1} \\
c)&\qquad\qquad\qquad\qquad&\psi\circ \psi&=-\frac{27}{2} Q\wedge_{\epsilon} Q \wedge_{\times} Id_V \quad &&\text{ in } Alt^9_{\epsilon}(V,V), \label{id mat2} \\
d)&\qquad\qquad\qquad\qquad& Q\circ \psi&=-54 Q\wedge_{\epsilon} Q \wedge_{\epsilon} Q \quad && \text{ in } Alt^{12}_{\epsilon}(V,k). \label{id mat3}
\end{alignat}
\end{thm}

\begin{proof} In this proof, for $v_1,\ldots,v_n \in V$ and $\sigma \in S_n$ we denote $p(\sigma ; v_1, \ldots, v_n)$ by $p(\sigma ; v)$ and we use implicitly Remarks \ref{rem formule de p} and \ref{rem shuffle}. Let $v_1,\ldots , v_{12}\in V$.
\vspace{0.3cm}

\noindent
$a)$ Since
$$\mu(v_1,v_2)=\frac{1}{2} \sum \limits_{\sigma\in S_2} p(\sigma ; v) \mu(v_{\sigma(1)},v_{\sigma(2)}),$$
it follows from Equation \eqref{formule Bpsi pour mathews} that
\begin{equation}\label{preuve mat formule mu wedge psi}
\mu\wedge_{\rho} \psi(v_1,v_2,v_3,v_4,v_5)=\frac{1}{4}\sum \limits_{\sigma\in S_5} p(\sigma ; v)\mu(v_{\sigma(1)},v_{\sigma(2)})(\mu(v_{\sigma(3)},v_{\sigma(4)})(v_{\sigma(5)})).
\end{equation}
Since $V$ is a special $\epsilon$-orthogonal representation, using Proposition \ref{equivalent conditions CS} \ref{CS}, we have
\begin{align}
&\mu(v_{\sigma(1)},v_{\sigma(2)})(\mu(v_{\sigma(3)},v_{\sigma(4)})(v_{\sigma(5)}))=\epsilon(v_{\sigma(1)},v_{\sigma(2)})\Big(v_{\sigma(1)},\mu(v_{\sigma(3)},v_{\sigma(4)})(v_{\sigma(5)})\Big)v_{\sigma(2)}\nonumber\\
&-2\Big(v_{\sigma(2)},\mu(v_{\sigma(3)},v_{\sigma(4)})(v_{\sigma(5)})\Big)v_{\sigma(1)}+(v_{\sigma(1)},v_{\sigma(2)})\mu(v_{\sigma(3)},v_{\sigma(4)})(v_{\sigma(5)})\nonumber\\
&-\epsilon(v_{\sigma(2)},v_{\sigma(3)}+v_{\sigma(4)}+v_{\sigma(5)})\mu(v_{\sigma(1)},\mu(v_{\sigma(3)},v_{\sigma(4)})(v_{\sigma(5)}))(v_{\sigma(2)})\label{preuve mat CS}.
\end{align}
When we take the $\epsilon$-antisymmetric sum over all permutations as in \eqref{preuve mat formule mu wedge psi} it turns out that the sums corresponding to the third and fourth terms in \eqref{preuve mat CS} vasnish as we now show.
\vspace{0.2cm}

Consider the subgroup $H:=\lbrace id,(12) \rbrace$ acting on the right on $S_5$. We have a partition $S_5=\cup \mathcal{O}_{\sigma}$ where $\mathcal{O}_{\sigma}$ is the orbit of the element $\sigma$ under the action of $H$, and there are $60$ orbits each containing $2$ elements. Since 
$$p(\sigma ; v)(v_{\sigma(1)},v_{\sigma(2)})\mu(v_{\sigma(3)},v_{\sigma(4)})(v_{\sigma(5)})+p(\sigma(12) ; v)(v_{\sigma(2)},v_{\sigma(1)})\mu(v_{\sigma(3)},v_{\sigma(4)})(v_{\sigma(5)})=0$$
then we have
\begin{equation} \label{preuve mat =0 1}
\sum \limits_{\sigma\in S_5} p(\sigma ; v)(v_{\sigma(1)},v_{\sigma(2)})\mu(v_{\sigma(3)},v_{\sigma(4)})(v_{\sigma(5)})=0.
\end{equation}
We now show that
\begin{equation} \label{preuve mat =0 2}
-\sum \limits_{\sigma\in S_5} p(\sigma ; v)\epsilon(v_{\sigma(2)},v_{\sigma(3)}+v_{\sigma(4)}+v_{\sigma(5)})\mu(v_{\sigma(1)},\mu(v_{\sigma(3)},v_{\sigma(4)})(v_{\sigma(5)}))(v_{\sigma(2)})=0.
\end{equation}
First of all, by setting $\sigma':=(2345)\in S_5$ we have
\begin{align}
&-\sum \limits_{\sigma\in S_5} p(\sigma ; v)\epsilon(v_{\sigma(2)},v_{\sigma(3)}+v_{\sigma(4)}+v_{\sigma(5)})\mu(v_{\sigma(1)},\mu(v_{\sigma(3)},v_{\sigma(4)})(v_{\sigma(5)}))(v_{\sigma(2)})\nonumber\\
&=-\sum \limits_{\sigma\in S_5} p(\sigma ; v)\epsilon(v_{\sigma\sigma'(5)},v_{\sigma\sigma'(2)}+v_{\sigma\sigma'(3)}+v_{\sigma\sigma'(4)})\mu(v_{\sigma\sigma'(1)},\mu(v_{\sigma\sigma'(2)},v_{\sigma\sigma'(3)})(v_{\sigma\sigma'(4)}))(v_{\sigma\sigma'(5)})\nonumber\\
&=\sum \limits_{\sigma\in S_5} p(\sigma ; v)\mu(v_{\sigma(1)},\mu(v_{\sigma(2)},v_{\sigma(3)})(v_{\sigma(4)}))(v_{\sigma(5)}).\label{preuve mat egalite reparam 1}
\end{align}
Moreover,
\begin{align}
&\sum \limits_{\sigma\in S_5} p(\sigma ; v)\mu(v_{\sigma(1)},\mu(v_{\sigma(2)},v_{\sigma(3)})(v_{\sigma(4)}))(v_{\sigma(5)})\nonumber\\
&=\sum \limits_{\sigma'\in S(\llbracket 1,4 \rrbracket,\lbrace 5\rbrace)} p(\sigma' ; v) \sum \limits_{\sigma \in S_4} \Big( p(\sigma ; v_{\sigma'})\mu(v_{\sigma\sigma'(1)},\mu(v_{\sigma\sigma'(2)},v_{\sigma\sigma'(3)})(v_{\sigma\sigma'(4)})) \Big)(v_{\sigma'(5)}). \label{preuve mat egalite reparam 1.2}
\end{align}
Consider the subgroup
$$H:=\lbrace id, (14),(12)(34),(1342),(23),(14)(23),(1243),(13)(24)\rbrace$$
acting on the right on $S_4$. We have a partition $S_4=\cup \mathcal{O}_{\sigma}$ where $\mathcal{O}_{\sigma}$ is the orbit of the element $\sigma$ under the action of $H$, and there are $3$ orbits each containing $8$ elements. Using the $\epsilon$-antisymmetry of $\mu$ we obtain
$$\sum \limits_{\sigma'\in \mathcal{O}_{\sigma}} p(\sigma' ; v)\mu(v_{\sigma'(1)},\mu(v_{\sigma'(2)},v_{\sigma'(3)})(v_{\sigma'(4)}))=2\sum \limits_{\sigma'\in \tilde{\mathcal{O}}_{\sigma}} p(\sigma' ; v)\mu(v_{\sigma'(1)},\mu(v_{\sigma'(2)},v_{\sigma'(3)})(v_{\sigma'(4)}))$$
where
$$\tilde{\mathcal{O}}_{\sigma}=\lbrace \sigma, \sigma(14),\sigma(12)(34),\sigma(1342) \rbrace.$$
We have
\begin{align*}
&\sum \limits_{\sigma'\in \tilde{\mathcal{O}}_{\sigma}} p(\sigma' ; v)\mu(v_{\sigma'(1)},\mu(v_{\sigma'(2)},v_{\sigma'(3)})(v_{\sigma'(4)}))\\
&=p(\sigma ; v)\mu(v_{\sigma(1)},\mu(v_{\sigma(2)},v_{\sigma(3)})(v_{\sigma(4)}))+p(\sigma(14) ; v)\mu(v_{\sigma(4)},\mu(v_{\sigma(2)},v_{\sigma(3)})(v_{\sigma(1)}))\\
&+p(\sigma (12)(34) ; v)\mu(v_{\sigma(2)},\mu(v_{\sigma(1)},v_{\sigma(4)})(v_{\sigma(3)}))+p(\sigma(1342) ; v)\mu(v_{\sigma(3)},\mu(v_{\sigma(1)},v_{\sigma(4)})(v_{\sigma(2)}))\\
&=p(\sigma ; v)\epsilon(v_{\sigma(1)},v_{\sigma(2)}+v_{\sigma(3)})[\mu(v_{\sigma(2)},v_{\sigma(3)}),\mu(v_{\sigma(1)},v_{\sigma(4)})]\\
&+p(\sigma ; v)\epsilon(v_{\sigma(2)}+v_{\sigma(3)},v_{\sigma(4)})[\mu(v_{\sigma(1)},v_{\sigma(4)}),\mu(v_{\sigma(2)},v_{\sigma(3)})]\\
&=0.
\end{align*}
Hence
\begin{equation} \label{preuve mat eq=0 intermediaire}
\sum \limits_{\sigma\in S_4} p(\sigma ; v)\mu(v_{\sigma(1)},\mu(v_{\sigma(2)},v_{\sigma(3)})(v_{\sigma(4)}))=0
\end{equation}
and using \eqref{preuve mat egalite reparam 1.2} and \eqref{preuve mat egalite reparam 1} this implies \eqref{preuve mat =0 2}. Therefore, from Equations \eqref{preuve mat CS} and \eqref{preuve mat =0 1}, \eqref{preuve mat =0 2} it follows that
\begin{align*}
&\sum \limits_{\sigma\in S_5} p(\sigma ; v)\mu(v_{\sigma(1)},v_{\sigma(2)})(\mu(v_{\sigma(3)},v_{\sigma(4)})(v_{\sigma(5)}))\\
&=\sum \limits_{\sigma\in S_5} p(\sigma ; v)\Big( \epsilon(v_{\sigma(1)},v_{\sigma(2)})(v_{\sigma(1)},\mu(v_{\sigma(3)},v_{\sigma(4)})(v_{\sigma(5)}))v_{\sigma(2)}-2(v_{\sigma(2)},\mu(v_{\sigma(3)},v_{\sigma(4)})(v_{\sigma(5)}))v_{\sigma(1)} \Big).
\end{align*}
But, if we set $\sigma'=(2345)$, we have
\begin{align*}
&\sum \limits_{\sigma\in S_5} p(\sigma ; v) \epsilon(v_{\sigma(1)},v_{\sigma(2)})\Big(v_{\sigma(1)},\mu(v_{\sigma(3)},v_{\sigma(4)})(v_{\sigma(5)})\Big)v_{\sigma(2)}\\
&=-\sum \limits_{\sigma\in S_5} p(\sigma \sigma' ; v) \Big(v_{\sigma\sigma'(1)},\mu(v_{\sigma\sigma'(2)},v_{\sigma\sigma'(3)})(v_{\sigma\sigma'(4)})\Big)v_{\sigma\sigma'(5)}\\
&=-\sum \limits_{\sigma\in S_5} p(\sigma ; v) \Big(v_{\sigma(1)},\mu(v_{\sigma(2)},v_{\sigma(3)})(v_{\sigma(4)})\Big)v_{\sigma(5)}.
\end{align*}
Similarly with $\sigma'=(12345)$ we have
$$-2\sum \limits_{\sigma\in S_5} p(\sigma ; v)\Big(v_{\sigma(2)},\mu(v_{\sigma(3)},v_{\sigma(4)})(v_{\sigma(5)})\Big)v_{\sigma(1)}=-2\sum \limits_{\sigma\in S_5} p(\sigma ; v) \Big(v_{\sigma(1)},\mu(v_{\sigma(2)},v_{\sigma(3)})(v_{\sigma(4)})\Big)v_{\sigma(5)}.$$
Therefore, we obtain
\begin{equation} \label{preuve mathews eq1}
\sum \limits_{\sigma\in S_5} p(\sigma ; v)\mu(v_{\sigma(1)},v_{\sigma(2)})(\mu(v_{\sigma(3)},v_{\sigma(4)})(v_{\sigma(5)}))=-3\sum \limits_{\sigma\in S_5} p(\sigma ; v) \Big(v_{\sigma(1)},\mu(v_{\sigma(2)},v_{\sigma(3)})(v_{\sigma(4)})\Big)v_{\sigma(5)}.
\end{equation}
Since by definition
\begin{equation}\label{preuve mat formule Q wedge Id}
Q \wedge_{\times} Id_V (v_1,v_2,v_3,v_4,v_5)=\frac{1}{2}\sum \limits_{\sigma\in S_5} p(\sigma ; v) (v_{\sigma(1)},\mu(v_{\sigma(2)},v_{\sigma(3)})(v_{\sigma(4)}))v_{\sigma(5)}
\end{equation}
it follows from Equations \eqref{preuve mat formule mu wedge psi} and \eqref{preuve mathews eq1} that
$$\mu\wedge_{\rho} \psi=-\frac{3}{2}Q \wedge_{\times} Id_V.$$
\vspace{0.2cm}

\noindent
$b)$ By definition
\begin{equation}\label{preuve mat formule mu circ psi}
\mu\circ \psi (v_1,v_2,v_3,v_4,v_5,v_6)=\frac{1}{4} \sum \limits_{\sigma\in S_6} p(\sigma ; v)\mu\left( \mu(v_{\sigma(1)},v_{\sigma(2)})(v_{\sigma(3)}),\mu(v_{\sigma(4)},v_{\sigma(5)})(v_{\sigma(6)}) \right).
\end{equation}
By equivariance, we have
\begin{align*}
&\sum \limits_{\sigma\in S_6} p(\sigma ; v)\mu\left( \mu(v_{\sigma(1)},v_{\sigma(2)})(v_{\sigma(3)}),\mu(v_{\sigma(4)},v_{\sigma(5)})(v_{\sigma(6)}) \right)\\
&=\sum \limits_{\sigma\in S_6} p(\sigma ; v)  \Big( \left[\mu(v_{\sigma(1)},v_{\sigma(2)}),\mu(v_{\sigma(3)},\mu(v_{\sigma(4)},v_{\sigma(5)})(v_{\sigma(6)}))\right]\\
&-\epsilon(v_{\sigma(1)}+v_{\sigma(2)},v_{\sigma(3)})\mu ( v_{\sigma(3)},\mu(v_{\sigma(1)},v_{\sigma(2)})(\mu(v_{\sigma(4)},v_{\sigma(5)})(v_{\sigma(6)})) ) \Big).
\end{align*}
Using \eqref{preuve mat eq=0 intermediaire}, we have
\begin{align*}
&\sum \limits_{\sigma\in S_6} p(\sigma ; v)  \left[\mu(v_{\sigma(1)},v_{\sigma(2)}),\mu(v_{\sigma(3)},\mu(v_{\sigma(4)},v_{\sigma(5)})(v_{\sigma(6)}))\right]\\
&=\sum \limits_{\sigma'\in S(\llbracket 1,2 \rrbracket,\llbracket 3,6 \rrbracket)} p(\sigma' ; v)  \left[\sum \limits_{\sigma \in S_2} p(\sigma; v_{\sigma'}) \mu(v_{\sigma\sigma'(1)},v_{\sigma\sigma'(2)}),\sum \limits_{\sigma \in S_4} p(\sigma; v_{\sigma'}) \mu(v_{\sigma\sigma'(3)},\mu(v_{\sigma\sigma'(4)},v_{\sigma\sigma'(5)})(v_{\sigma\sigma'(6)}))\right]\\
&=0.
\end{align*}
Thus 
\begin{align*}
&\sum \limits_{\sigma\in S_6} p(\sigma ; v)\mu\left( \mu(v_{\sigma(1)},v_{\sigma(2)})(v_{\sigma(3)}),\mu(v_{\sigma(4)},v_{\sigma(5)})(v_{\sigma(6)}) \right)\\
&=-\sum \limits_{\sigma\in S_6} p(\sigma ; v)\epsilon(v_{\sigma(1)}+v_{\sigma(2)},v_{\sigma(3)})\mu ( v_{\sigma(3)},\mu(v_{\sigma(1)},v_{\sigma(2)})(\mu(v_{\sigma(4)},v_{\sigma(5)})(v_{\sigma(6)})))
\end{align*}
and with a change of index using $\sigma':=(132)$ this is equal to
$$-\sum \limits_{\sigma\in S_6} p(\sigma ; v)\mu ( v_{\sigma(1)},\mu(v_{\sigma(2)},v_{\sigma(3)})(\mu(v_{\sigma(4)},v_{\sigma(5)})(v_{\sigma(6)})))$$
and this is equal to
$$=-\sum \limits_{\sigma'\in S(\llbracket 1 \rrbracket,\llbracket 2,6 \rrbracket)} \Big( p(\sigma' ; v)\mu ( v_{\sigma'(1)},\sum \limits_{\sigma\in S_5} p(\sigma ; v_{\sigma'}) \mu(v_{\sigma\sigma'(2)},v_{\sigma\sigma'(3)})(\mu(v_{\sigma\sigma'(4)},v_{\sigma\sigma'(5)})(v_{\sigma\sigma'(6)})))\Big).$$
Using Equation \eqref{preuve mathews eq1}, we have
\begin{align}
&\sum \limits_{\sigma\in S_6} p(\sigma ; v)\mu\left( \mu(v_{\sigma(1)},v_{\sigma(2)})(v_{\sigma(3)}),\mu(v_{\sigma(4)},v_{\sigma(5)})(v_{\sigma(6)}) \right)\nonumber\\
&=3\sum \limits_{\sigma'\in S(\llbracket 1 \rrbracket,\llbracket 2,6 \rrbracket)} \Big( p(\sigma' ; v)\mu ( v_{\sigma'(1)}, \sum \limits_{\sigma\in S_5} p(\sigma ; v_{\sigma'}) (v_{\sigma\sigma'(2)},\mu(v_{\sigma\sigma'(3)},v_{\sigma\sigma'(4)})(v_{\sigma\sigma'(5)}))v_{\sigma\sigma'(6)} \Big)\nonumber\\
&=3\sum \limits_{\sigma\in S_6} \Big( p(\sigma ; v)\mu ( v_{\sigma(1)}, (v_{\sigma(2)},\mu(v_{\sigma(3)},v_{\sigma(4)})(v_{\sigma(5)}))v_{\sigma(6)}) \Big)\nonumber\\
&=3\sum \limits_{\sigma\in S_6} \Big( p(\sigma ; v)(v_{\sigma(1)},\mu(v_{\sigma(2)},v_{\sigma(3)})(v_{\sigma(4)}))\mu ( v_{\sigma(5)}, v_{\sigma(6)}) \Big). \label{preuve mathews eq2}
\end{align}
Since by definition
\begin{equation} \label{preuve mat formule Q wedge mu}
Q\wedge_{\times}\mu(v_1,v_2,v_3,v_4,v_5,v_6)=\frac{1}{4}\sum \limits_{\sigma\in S_6} p(\sigma ; v)(v_{\sigma(1)},\mu(v_{\sigma(2)},v_{\sigma(3)})(v_{\sigma(4)}))\mu(v_{\sigma(5)},v_{\sigma(6)})
\end{equation}
it follows from Equations \eqref{preuve mat formule mu circ psi} and \eqref{preuve mathews eq2} that
$$\mu\circ \psi=3Q\wedge_{\times} \mu.$$
\vspace{0.2cm}

\noindent
$c)$ By definition
\begin{align}
&\psi\circ \psi(v_1,v_2,v_3,v_4,v_5,v_6,v_7,v_8,v_9)\nonumber\\
&=\frac{3}{8}\sum \limits_{\sigma\in S_9} p(\sigma ; v) \mu\Big( \mu(v_{\sigma(1)},v_{\sigma(2)})(v_{\sigma(3)}),\mu(v_{\sigma(4)},v_{\sigma(5)})(v_{\sigma(6)}) \Big)(\mu(v_{\sigma(7)},v_{\sigma(8)})(v_{\sigma(9)})). \label{preuve mat formule psi circ psi}
\end{align}
We have
\begin{align*}
&\sum \limits_{\sigma\in S_9} p(\sigma ; v) \mu\Big( \mu(v_{\sigma(1)},v_{\sigma(2)})(v_{\sigma(3)}),\mu(v_{\sigma(4)},v_{\sigma(5)})(v_{\sigma(6)}) \Big)(\mu(v_{\sigma(7)},v_{\sigma(8)})(v_{\sigma(9)}))\\
&=\sum \limits_{\sigma' \in S(\llbracket 1,6 \rrbracket,\llbracket 7,9 \rrbracket)} p(\sigma' ; v) \sum \limits_{\sigma \in S_6} p(\sigma;v_{\sigma'}) \mu\Big( \mu(v_{\sigma\sigma' (1)},v_{\sigma\sigma' (2)})(v_{\sigma\sigma' (3)}),\mu(v_{\sigma\sigma' (4)},v_{\sigma\sigma' (5)})(v_{\sigma\sigma' (6)}) \Big)\\
&\Big(\sum \limits_{\sigma \in S_3}p(\sigma;v_{\sigma'}) \mu(v_{\sigma\sigma' (7)},v_{\sigma\sigma' (8)})(v_{\sigma\sigma' (9)})\Big)
\end{align*}
and using Equation \eqref{preuve mathews eq2} this is equal to
\begin{align*}
&3\sum \limits_{\sigma \in S_9} p(\sigma ; v) (v_{\sigma(1)},\mu(v_{\sigma(2)},v_{\sigma(3)})(v_{\sigma(4)}))\mu(v_{\sigma(5)},v_{\sigma(6)}) (\mu(v_{\sigma(7)},v_{\sigma(8)})(v_{\sigma(9)}))\\
&=3\sum \limits_{\sigma' \in S(\llbracket 1,4 \rrbracket, \llbracket 5,9 \rrbracket)} p(\sigma' ; v) \sum \limits_{\sigma \in S_4}p(\sigma ; v_{\sigma'}) (v_{\sigma\sigma'(1)},\mu(v_{\sigma\sigma'(2)},v_{\sigma\sigma'(3)})(v_{\sigma\sigma'(4)}))\\
&\sum \limits_{\sigma \in S_5} p(\sigma ; v_{\sigma'})\mu(v_{\sigma\sigma'(5)},v_{\sigma\sigma'(6)})( \mu(v_{\sigma\sigma'(7)},v_{\sigma\sigma'(8)})(v_{\sigma\sigma'(9)})).
\end{align*}
Using Equation \eqref{preuve mathews eq1} we have
\begin{align}
&\sum \limits_{\sigma\in S_9} p(\sigma ; v) \mu\Big( \mu(v_{\sigma(1)},v_{\sigma(2)})(v_{\sigma(3)}),\mu(v_{\sigma(4)},v_{\sigma(5)})(v_{\sigma(6)}) \Big)(\mu(v_{\sigma(7)},v_{\sigma(8)})(v_{\sigma(9)}))\nonumber\\
&=-9\sum \limits_{\sigma' \in S(\llbracket 1,4 \rrbracket, \llbracket 5,9 \rrbracket)} p(\sigma' ; v) \sum \limits_{\sigma \in S_4}p(\sigma ; v_{\sigma'}) (v_{\sigma\sigma'(1)},\mu(v_{\sigma\sigma'(2)},v_{\sigma\sigma'(3)})(v_{\sigma\sigma'(4)}))\nonumber\\
&\sum \limits_{\sigma\in S_5} p(\sigma ; v_{\sigma'}) (v_{\sigma\sigma'(5)},\mu(v_{\sigma\sigma'(6)},v_{\sigma\sigma'(7)})(v_{\sigma\sigma'(8)}))v_{\sigma\sigma'(9)}\nonumber\\
&=-9\sum \limits_{\sigma \in S_9} p(\sigma ; v) (v_{\sigma(1)},\mu(v_{\sigma(2)},v_{\sigma(3)})(v_{\sigma(4)}))(v_{\sigma(5)},\mu(v_{\sigma(6)},v_{\sigma(7)})(v_{\sigma(8)}))v_{\sigma(9)}. \label{preuve mathews eq3}
\end{align}
Since by definition
\begin{align}
&Q\wedge_{\epsilon} Q\wedge_{\times} Id_V (v_1,v_2,v_3,v_4,v_5,v_6,v_7,v_8,v_9)\nonumber\\
&=\frac{1}{4}\sum \limits_{\sigma \in S_9} p(\sigma ; v) (v_{\sigma(1)},\mu(v_{\sigma(2)},v_{\sigma(3)})(v_{\sigma(4)}))(v_{\sigma(5)},\mu(v_{\sigma(6)},v_{\sigma(7)})(v_{\sigma(8)}))v_{\sigma(9)}, \label{preuve mat formule Q wedge Q wedge Id}
\end{align}
it follows from Equations \eqref{preuve mat formule psi circ psi} and \eqref{preuve mathews eq3} that
$$\psi\circ \psi=-\frac{27}{2}Q\wedge_{\epsilon} Q\wedge_{\times} Id_V.$$
\vspace{0.2cm}

\noindent
$d)$ By definition
\begin{align}
&Q\circ \psi(v_1,v_2,v_3,v_4,v_5,v_6,v_7,v_8,v_9,v_{10},v_{11},v_{12})\nonumber\\
&=\frac{3}{4}\sum \limits_{\sigma \in S_{12}} p(\sigma ; v) \Big( \mu(v_{\sigma(1)},v_{\sigma(2)})(v_{\sigma(3)}),\mu(\mu(v_{\sigma(4)},v_{\sigma(5)})(v_{\sigma(6)}),\mu(v_{\sigma(7)},v_{\sigma(8)})(v_{\sigma(9)}))(\mu(v_{\sigma(10)},v_{\sigma(11)})(v_{\sigma(12)})) \Big). \label{preuve mat formule Q circ psi}
\end{align}
We have
\begin{align*}
&\sum \limits_{\sigma \in S_{12}} p(\sigma ; v) \Big( \mu(v_{\sigma(1)},v_{\sigma(2)})(v_{\sigma(3)}),\mu(\mu(v_{\sigma(4)},v_{\sigma(5)})(v_{\sigma(6)}),\mu(v_{\sigma(7)},v_{\sigma(8)})(v_{\sigma(9)}))(\mu(v_{\sigma(10)},v_{\sigma(11)})(v_{\sigma(12)})) \Big)\\
&=\sum \limits_{\sigma' \in S(\llbracket1,3\rrbracket,\llbracket4,12\rrbracket)} p(\sigma' ; v) \Big( \sum \limits_{\sigma \in S_3} p(\sigma,v_{\sigma'}) \mu(v_{\sigma\sigma'(1)},v_{\sigma\sigma'(2)})(v_{\sigma\sigma'(3)}),\\
& \sum \limits_{\sigma \in S_9} p(\sigma,v_{\sigma'})  \mu\Big( \mu(v_{\sigma\sigma'(4)},v_{\sigma\sigma'(5)})(v_{\sigma\sigma'(6)}),\mu(v_{\sigma\sigma'(7)},v_{\sigma\sigma'(8)})(v_{\sigma\sigma'(9)}) \Big)(\mu(v_{\sigma\sigma'(10)},v_{\sigma\sigma'(11)})(v_{\sigma\sigma'(12)}))\Big)
\end{align*}
and using \eqref{preuve mathews eq3} this is equal to
\begin{align*}
&=-9\sum \limits_{\sigma \in S_{12}} p(\sigma ; v)(\mu(v_{\sigma(1)},v_{\sigma(2)})(v_{\sigma(3)}),v_{\sigma(12)})(v_{\sigma(4)},\mu(v_{\sigma(5)},v_{\sigma(6)})(v_{\sigma(7)}))(v_{\sigma(8)},\mu(v_{\sigma(9)},v_{\sigma(10)})(v_{\sigma(11)}))\\
&=-9\sum \limits_{\sigma \in S_{12}} p(\sigma ; v)\epsilon(v_{\sigma(1)}+v_{\sigma(2)}+v_{\sigma(3)},v_{\sigma(12)})\\
&(v_{\sigma(12)},\mu(v_{\sigma(1)},v_{\sigma(2)})(v_{\sigma(3)}))(v_{\sigma(4)},\mu(v_{\sigma(5)},v_{\sigma(6)})(v_{\sigma(7)}))(v_{\sigma(8)},\mu(v_{\sigma(9)},v_{\sigma(10)})(v_{\sigma(11)}))\\
&=-9\sum \limits_{\sigma \in S_{12}} p(\sigma ; v)(v_{\sigma(1)},\mu(v_{\sigma(2)},v_{\sigma(3)})(v_{\sigma(4)}))(v_{\sigma(5)},\mu(v_{\sigma(6)},v_{\sigma(7)})(v_{\sigma(8)}))(v_{\sigma(9)},\mu(v_{\sigma(10)},v_{\sigma(11)})(v_{\sigma(12)})).
\end{align*}
Since by definition
\begin{align}
&Q\wedge_{\epsilon} Q \wedge_{\epsilon} Q(v_1,v_2,v_3,v_4,v_5,v_6,v_7,v_8,v_9,v_{10},v_{11},v_{12})\nonumber\\
&=\frac{1}{8}\sum \limits_{\sigma \in S_{12}} p(\sigma ; v)(v_{\sigma(1)},\mu(v_{\sigma(2)},v_{\sigma(3)})(v_{\sigma(4)}))(v_{\sigma(5)},\mu(v_{\sigma(6)},v_{\sigma(7)})(v_{\sigma(8)}))(v_{\sigma(9)},\mu(v_{\sigma(10)},v_{\sigma(11)})(v_{\sigma(12)})),\label{preuve mat formule Q wedge Q wedge Q}
\end{align}
it follows from Equation \eqref{preuve mat formule Q wedge Q wedge Q} that
$$Q\circ \psi=-54 Q\wedge_{\epsilon} Q \wedge_{\epsilon} Q.$$
\end{proof}

\newpage

\section{Appendix}

Let $\Gamma$ be an abelian group and let $\epsilon$ be a commutation factor of $\Gamma$. In this section, we give examples of special $\epsilon$-orthogonal representations. For proofs and details see \cite{MeyerThesis}.
\vspace{0.2cm}

Throughout this section we suppose that the representation $k^2$ of $\sl(2,k)$ is an $\epsilon$-orthogonal representation with respect to $(\Gamma,\epsilon)$.

\subsection{The fundamental representation of $\mathfrak{so}_{\epsilon}(V,( ~ , ~ ))$} \label{subsec fund rep of so}

Let $V$ be a finite-dimensional $\Gamma$-graded vector space and let $(\phantom{v},\phantom{v})$ be a non-degenerate $\epsilon$-symmetric bilinear form on $V$. The colour Lie algebra $\so_{\epsilon}(V,(\phantom{v},\phantom{v}))$ is $\epsilon$-quadratic for the bilinear form
$$B(f,g)=-\frac{1}{2}Tr(\mathcal{E}\circ f\circ g) \qquad \forall f,g \in \so_{\epsilon}(V,(\phantom{v},\phantom{v})),$$
and the fundamental representation of $\so_{\epsilon}(V,(\phantom{v},\phantom{v}))$ has moment map $\mu_{can}$ (see Proposition \ref{pp moment map canonique}) which trivially satisfies \ref{CS 2} of Proposition \ref{equivalent conditions CS}. Hence the fundamental representation of $\so_{\epsilon}(V,(\phantom{v},\phantom{v}))$ is a special $\epsilon$-orthogonal representation and by Theorem \ref{rep tensor prod with sl2} there is a colour Lie algebra of the form
$$\so_{\epsilon}(V,(\phantom{v},\phantom{v}))\oplus \sl(2,k)\oplus V\otimes k^2.$$
One can check that it is isomorphic to the colour Lie algebra $\so_{\epsilon}(V\oplus k^2,(\phantom{v},\phantom{v})\perp \omega)$ and that the associated covariants $\psi \in Alt^3_{\epsilon}(V,V)$ and $Q \in Alt_{\epsilon}^4(V)$ vanish.

\subsection{The fundamental representation of $\mathfrak{so}_{\epsilon}(V,(~,~))\oplus \mathfrak{sl}(2,k)$}

Let $V$ be a finite-dimensional $\Gamma$-graded vector space, let $(\phantom{v},\phantom{v})$ be a non-degenerate $\epsilon$-symmetric bilinear form on $V$ and let $(W,\Omega)$ be a two-dimensional $\Gamma$-graded symplectic vector space. The colour Lie algebra $\so_{\epsilon}(V,(\phantom{v},\phantom{v}))\oplus \so_{\epsilon}(W,\Omega)$ is $\epsilon$-quadratic for the bilinear form $B_{V}\perp B_{W}$ where
\begin{alignat*}{3}
&B_V(f,g)&&=\frac{1}{4}Tr(\mathcal{E}\circ f \circ g) \quad &&\forall f,g \in\so_{\epsilon}(V,(\phantom{v},\phantom{v})), \\
&B_W(f,g)&&=-\frac{1}{2}Tr(\mathcal{E}\circ f \circ g) \quad &&\forall f,g \in\so_{\epsilon}(W,\Omega).
\end{alignat*}
The $\epsilon$-orthogonal representation
$$\so_{\epsilon}(V,(\phantom{v},\phantom{v}))\oplus \so_{\epsilon}(W,\Omega)\rightarrow \so_{\epsilon}(V\otimes W,(\phantom{v},\phantom{v})\otimes \Omega)$$
of the $\epsilon$-quadratic colour Lie algebra $\Big(\so_{\epsilon}(V,(\phantom{v},\phantom{v}))\oplus \so_{\epsilon}(W,\Omega),B_V\perp B_W\Big)$ has moment map which satisfies \ref{CS} of Proposition \ref{equivalent conditions CS} and then by Theorem \ref{rep tensor prod with sl2} it gives rise to colour Lie algebra of the form
$$\so_{\epsilon}(V,(\phantom{v},\phantom{v}))\oplus \so_{\epsilon}(W,\Omega)\oplus\sl(2,k)\oplus V\otimes W\otimes k^2.$$
One can check that it is isomorphic to $\so_{\epsilon}(V\oplus H,(\phantom{v},\phantom{v})\perp(\phantom{v},\phantom{v})_H)$ where $(H,(\phantom{v},\phantom{v})_H)$ is a four-dimensional hyperbolic vector space.
\vspace{0.2cm}

Let $\lbrace p,q\rbrace$ be an homogeneous basis of $W$ such that $\Omega(p,q)=1$ and let $Q \in Alt_{\epsilon}^4(V\otimes W)$ be the quadrilinear covariant of the $\epsilon$-orthogonal representation $V\otimes W$. For all $v_1,v_2,v_3,v_4 \in V$ we have
\begin{align*}
Q(v_1\otimes p,v_2\otimes p,v_3\otimes p,v_4\otimes p)&=Q(v_1\otimes q,v_2\otimes q,v_3\otimes q,v_4\otimes q)=0,\\
Q(v_1\otimes p,v_2\otimes p,v_3\otimes p,v_4\otimes q)&=Q(v_1\otimes p,v_2\otimes q,v_3\otimes q,v_4\otimes q)=0
\end{align*}
and
\begin{align*}
&Q(v_1\otimes p,v_2\otimes p,v_3\otimes q,v_4\otimes q)\\
&=24\epsilon(p,v_2+v_3)(v_1,v_2)(v_3,v_4)-12\epsilon(p,v_3)^2\epsilon(v_3,v_4)(v_1,v_3)(v_2,v_4)-12\epsilon(p,v_3+v_4)(v_2,v_3)(v_1,v_4).
\end{align*}

\subsection{Restriction of the fundamental representation of $\mathfrak{so}_{\epsilon}(V,( ~ , ~ ))$ to $\mathfrak{gl}_{\epsilon}(V)$ and $\mathfrak{u}_{\epsilon}(V,H)$}

Let $V$ be a finite-dimensional $\Gamma$-graded vector space such that $dim(V_0)\not\equiv dim(V_1)\mod char(k)$ and let $(\phantom{v},\phantom{v})$ be a non-degenerate $\epsilon$-symmetric bilinear form on $V$. Let $J \in \so_{\epsilon}(V,(\phantom{v},\phantom{v}))$ be such that $|J|=0$ and $J^2=\lambda Id$, where $\lambda \in k^*$. Let
$$\mm:=\lbrace f \in \so_{\epsilon}(V,(\phantom{v},\phantom{v})) ~ | ~ f\circ J=J\circ f \rbrace.$$
This is a colour Lie subalgebra of $\so_{\epsilon}(V,(\phantom{v},\phantom{v}))$.
\vspace{0.1cm}

There exists an $\epsilon$-symmetric, ad-invariant and non-degenerate bilinear form $B$ on $\mm$ such that the moment map of the $\epsilon$-quadratic representation $\rho : \mm \rightarrow \so_{\epsilon}(V,(\phantom{v},\phantom{v}))$ satisfies
\begin{equation}\label{moment 3 appendice}
\mu(v,w)=\mu_{can}(v,w)-\frac{1}{\lambda}\mu_{can}(J(v),J(w))+\frac{1}{\lambda}(J(v),w)J \qquad \forall v,w \in V.
\end{equation}
This moment map satisfies \ref{CS} of Proposition \ref{equivalent conditions CS} and then by Theorem \ref{rep tensor prod with sl2}, the representation $\mm\rightarrow \so_{\epsilon}(V,(\phantom{v},\phantom{v}))$ gives rise to a colour Lie algebra of the form
$$\gg:=\mm\oplus \sl(2,k)\oplus V\otimes k^2.$$
The covariant $\psi \in Alt^3_{\epsilon}(V,V)$ of the representation $\mm\rightarrow \so_{\epsilon}(V,(\phantom{v},\phantom{v}))$ satisfies: for all $v_1,v_2,v_3 \in V$,
$$\psi(v_1,v_2,v_3)=\frac{3}{\lambda}\Big((J(v_1),v_2)J(v_3)+\epsilon(v_1+v_2,v_3)(J(v_3),v_1)J(v_2)+(J(v_2),v_3)J(v_1)\Big).$$
\vspace{0.1cm}

\begin{itemize}
\item If $\lambda\in k^{*2}$, we have $V=W\oplus W'$ where $W$ (resp. $W'$) is the eigenspace of $J$ for the eigenvalue $\sqrt{\lambda}$ (resp. $-\sqrt{\lambda}$). The colour Lie algebra $\mm$ is isomorphic to $\gl_{\epsilon}(W)$ and $\gg$ is isomorphic to
$$
\left\{
    \begin{array}{ll}
        \sl_{\epsilon}(W\oplus k^2)& \mbox{if } dim(W_1)\not\equiv dim(W_0)-2\mod char(k); \\
        \gl_{\epsilon}(W\oplus k^2)/\mathcal{Z}(\gl_{\epsilon}(W\oplus k^2)) & \mbox{otherwise.}
    \end{array}
\right.
$$
\item If $\lambda\notin k^{*2}$, let $\kt=k(\sqrt{\lambda})$. We have $V\otimes \kt=W\oplus W'$ where $W$ (resp. $W'$) is the eigenspace of $J$ for the eigenvalue $\sqrt{\lambda}$ (resp. $-\sqrt{\lambda}$). The map $H : W\times W  \rightarrow \kt$ defined by
$$H(v,w):=(J(v),\bar{w}) \qquad \forall v,w\in W$$
is an $\epsilon$-antihermitian form and the colour Lie algebra $\mm$ is isomorphic to $\uu_{\epsilon}(W,H)$. If we consider the $\Gamma$-graded vector space $\kt^2=k^2\underset{k}{\otimes} \kt$ together with the $\epsilon$-antihermitian form $\Omega : \kt^2\times\kt^2\rightarrow \kt$ defined by
$$\Omega\Big(\begin{pmatrix}
a \\
b
\end{pmatrix},\begin{pmatrix}
c \\
d
\end{pmatrix}\Big):=a\bar{c}-b\bar{d} \qquad \forall a,b,c,d\in \kt,$$
then we have $\su_{\epsilon}(\kt^2,\Omega)\cong \sl(2,k)$. The colour Lie algebra $\gg$ is isomorphic to
$$
\left\{
    \begin{array}{ll}
       \su_{\epsilon}(W\oplus \kt^2,H\perp \Omega)& \mbox{if } dim(W_1)\not\equiv dim(W_0)-2\mod char(k); \\
        \uu_{\epsilon}(W\oplus \kt^2,H\perp \Omega)/\mathcal{Z}(\uu_{\epsilon}(W\oplus \kt^2,H\perp \Omega)) & \mbox{otherwise.}
    \end{array}
\right.
$$
\end{itemize}

\begin{rem}
The moment map \eqref{moment 3 appendice} takes its values in $\lbrace \mm,\mm\rbrace$ if and only if  $dim(W_1)\equiv dim(W_0)-2\mod char(k)$. In that case $V$ is a special $\epsilon$-orthogonal representation of the $\epsilon$-quadratic colour Lie algebra $(\lbrace\mm,\mm\rbrace,B)$ and the associated colour Lie algebra $\gg$ is isomorphic to
$$
\left\{
    \begin{array}{ll}
    \mathfrak{psl}_{\epsilon}(W\oplus k^2)& \mbox{if }\lambda\in k^{*2}; \\
       \mathfrak{psu}_{\epsilon}(W\oplus \kt^2,H\perp \Omega)& \mbox{otherwise.}
    \end{array}
\right.
$$
\end{rem}

\newpage

\footnotesize
\bibliographystyle{alpha}
\bibliography{generaliseKostant}
\addcontentsline{toc}{section}{\protect\numberline{}References} 

\end{document}